\documentclass[11pt]{amsart}

\usepackage{amscd,amstext,amsfonts,amsthm,amsxtra,amsmath}
\usepackage{amssymb}
\usepackage{fancyhdr}
\usepackage{mathrsfs}
\usepackage{enumerate}
\usepackage{array}
\usepackage{hyperref}
\usepackage[all]{xy}
\usepackage[onehalfspacing]{setspace}

\numberwithin{equation}{section}
\theoremstyle{plain}
\newtheorem{Main theorem}[equation]{Main Theorem}
\newtheorem{theorem}[equation]{Theorem}
\newtheorem{lemma}[equation]{Lemma}
\newtheorem{proposition}[equation]{Proposition}

\theoremstyle{definition}
\newtheorem{definition}[equation]{Definition}
\newtheorem{example}[equation]{Example}

\theoremstyle{plain}
\newtheorem{remark}[equation]{Remark}

\def\Aut{\operatorname{Aut}}
\def\aut{\operatorname{aut}}
\def\Hom{\operatorname{Hom}}
\def\End{\operatorname{End}}
\def\Sym{\operatorname{Sym}}
\def\Ker{\operatorname{Ker}}
\def\dim{\operatorname{dim}}
\def\Symb{\operatorname{Symb}}
\def\Gr{\operatorname{Gr}}
\def\Spin{\operatorname{Spin}}
\def\Sp{\operatorname{Sp}}
\def\SO{\operatorname{SO}}
\def\ad{\operatorname{ad}}
\def\Ad{\operatorname{Ad}}
\def\Sing{\operatorname{Sing}}
\def\rank{\operatorname{rank}}

\title[Geometric structures modeled on horospherical varieties]{Geometric structures modeled on\\ smooth projective horospherical varieties \\ of Picard number one}
\author[S. Kim]{Shin-young Kim}
\thanks{The author is supported by National Researcher Program 2010-0020413 of NRF}
\date{\today. Received May 26, 2015. Accepted October 24, 2016.}

\address{Korea Institute for Advanced Study, 85 Hoegiro, Dongdaemun-gu, Seoul, Korea}
\email{shinyoungkim@kias.re.kr}

\begin{document}

\maketitle

\begin{abstract}Geometric structures modeled on rational homogeneous manifolds are studied to characterize rational homogeneous manifolds and to prove their deformation rigidity. To generalize these characterizations and deformation rigidity results to quasihomogeneous varieties, we first study horospherical varieties and geometric structures modeled on horospherical varieties. Using Cartan geometry, we prove that a geometric structure modeled on a smooth projective horospherical variety of Picard number one is locally equivalent to the standard geometric structure when the geometric structure is defined on a Fano manifold of Picard number one.
\end{abstract}

\section*{Introduction}

 Let $M$ be a Fano manifold of Picard number one. An irreducible component $\mathcal K$ of the space of rational curves on $M$ is called a \emph{minimal dominating component} if the subvariety $\mathcal K_x$ consisting of members that pass through $x$ is nonempty and projective for a general point $x \in M$. The tangent directions at $x$ of members of $\mathcal K_x$ form a subvariety $\mathcal C_x$ of $\mathbb P T_x(M)$, which is called the \emph{variety of minimal rational tangents at $x$}. Many techniques can be used to study the projective geometries of $\mathcal C_x \subset \mathbb P T_x(M)$ which are believed to control the geometry of the manifold $M$. In this paper, we study geometric structures modeled on horospherical varieties which we expect to get from the variety of minimal rational tangents.

 When $S$ is a rational homogeneous manifold of Picard number one, a pair of the automorphism group of the variety of minimal rational tangent $\mathcal C_s$ and the linear span $D_s$ of the cone $\widehat{\mathcal C}_s\subset T_s(S)$ of $\mathcal C_s$ for $s\in S$ corresponds to the standard geometric structure on $S$. Jun-Muk Hwang, Ngaiming Mok, and Jaehyun Hong published significant work on the geometric structures modeled on $S$ that arise from the variety of minimal rational tangents. They published work on Hermitian symmetric manifolds and homogeneous contact manifolds in \cite{HM97}, \cite{Ho}, and \cite{Mok}, and on other rational homogeneous manifolds associated with long simple roots in \cite{HH}.

 \begin{theorem}[\cite{HM97}, \cite{Mok} and \cite{HH}]\label{Motivation1} Let $S=G/P$ where $G$ is a simple Lie group and $P$ is a maximal parabolic subgroup associated with a long root. Let $\mathcal C_s\subset \mathbb P T_s(S)$ be the variety of minimal rational tangents at a base point $s\in S$. Let $M$ be a Fano manifold of Picard number one and $\mathcal C_x$ be the variety of minimal rational tangents at a general point $x\in M$ associated with a minimal dominating rational component $\mathcal K$. Suppose that $\mathcal C_s\subset \mathbb P T_s(S)$ and $\mathcal C_x\subset \mathbb P T_x(M)$ are isomorphic as projective subvarieties for a general point $x \in M$. Then, $M$ is biholomorphic to $S$. \end{theorem}

 It is natural to ask what happens when we replace rational homogeneous manifolds with quasihomogeneous varieties, especially with smooth projective horospherical varieties of Picard number one. Horospherical varieties are complex normal algebraic varieties where a connected complex reductive algebraic group acts with an open orbit isomorphic to a torus bundle over a flag variety. Boris Pasquier classified smooth projective horospherical varieties of Picard number one in his paper \cite{Pa}. When a smooth horospherical variety is homogeneous, it is isomorphic to one of quadrics $Q^{2m}$, Grassmannians $\Gr(i+1,m+2)$, and spinor varieties $\Spin(2m+1)/P_{\alpha_m}$. These are compact irreducible Hermitian symmetric manifolds, and the geometric structures modeled on them have already been studied in Theorem \ref{Motivation1}.

 In this paper, we will study geometric structures modeled on smooth nonhomogeneous projective horospherical varieties of Picard number one.

 \begin{Main theorem}[Theorem \ref{Main Theorem}]Let $X$ be a smooth nonhomogeneous projective horospherical variety of Picard number one. Let $M$ be a Fano manifold of Picard number one. Then, any geometric structure of $M$ modeled on $X$ is locally equivalent to the standard geometric structure on $X$.
\end{Main theorem}

 We use Definition \ref{geometric ST} for the definition of a geometric structure modeled on $X$. We will prove the existence of Cartan connections (Proposition \ref{Cartan connection}) and use it to prove local equivalence of geometric structures modeled on smooth nonhomogeneous projective horospherical varieties of Picard number one.

 Noboru Tanaka (\cite{Ta70}, \cite{Ta79}) and Tohru Morimoto (\cite{Mo}) found the sufficient conditions for the existence of Cartan connections associated with geometric structures having certain symmetries, such as geometric structures modeled on rational homogeneous manifolds. We generalize these conditions for some quasihomogeneous manifold cases in Theorem \ref{connection}.

 To prove the existence of Cartan connections associated with geometric structures modeled on $X$, we need to study the Lie algebra $\frak{g}:=\frak{aut}(X)$ of the automorphism group of $X$. In particular, it is important to know whether $\frak g$ satisfies the prolongation property. Keizo Yamaguchi showed that $\frak g$ satisfies the prolongation property associated with a fundamental graded Lie subalgebra $\frak m$ of $\frak g$ by proving that the Lie algebra cohomology space $H^{p,1}(\frak m, \frak g)$ vanishes. When $X$ is a rational homogeneous manifold, $\frak g$ is a semisimple Lie algebra, and thus we can apply Kostant's harmonic theory to the Lie algebra cohomology spaces. However, in our case, $\frak g$ is not semisimple and we cannot apply Kostant's harmonic theory directly.

 Colleen Robles and Dennis The (\cite{RT}) computed Lie algebra cohomology spaces for some cases in which $\frak g$ is not semisimple by modifying Kostant's harmonic theory. In some ways, our embedding of $\frak{g}$ into $\frak{gl}(V)$ is similar to the embedding of a parabolic subalgebra into a simple graded Lie algebra of depth one, which appears in their paper \cite{RT}. It would be interesting if one can generalize Kostant's harmonic theory fully to the case where $\frak g$ is not semisimple. Instead of generalizing the whole of Kostant's harmonic theory, we reduce the vanishing of Lie algebra cohomology spaces to the vanishing of Lie algebra cohomology spaces associated with a maximal semisimple subalgebra of $\frak g$, which now can be computed using Kostant's harmonic theory.

 The remainder of this paper is organized as follows. In Section \ref{section1}, we review the theory of Cartan connections. In Section \ref{section2}, we prove that $\frak{aut}(X)$ satisfies the conditions for the existence of Cartan connections except the condition that the first generalized Spencer cohomology vanishes. In Section \ref{section3}, we compute the vanishing of the first generalized Spencer cohomology spaces of $\frak{aut}(X)$. In Section \ref{section4}, we prove the local flatness of the geometric structure modeled on $X$ and complete the proof of our main theorem.

 We work over the complex number field $\mathbb C$ without any additional mention of a number field. All manifolds, Lie groups, and Lie algebras will be understood as complex manifolds, complex Lie groups, and complex Lie algebras, respectively.

\section{Prolongations and Cartan connections}\label{section1}

\subsection{Prolongations}

\begin{definition}\label{gradation} Let $\frak g$ be a Lie algebra. A \emph{gradation of $\frak g$} is a direct decomposition $\frak g=\bigoplus_{p\in \mathbb Z} \frak g_p$ such that $[\frak g_p, \frak g_q]\subset \frak g_{p+q} \text{ for } p, q\in\mathbb Z$. A \emph{fundamental graded Lie algebra} is a nilpotent graded Lie algebra $\frak m=\bigoplus_{p<0} \frak g_p$ generated by $\frak g_{-1}$, i.e., $\frak g_p=[\frak g_{p+1}, \frak g_{-1}]$ for $p<-1$. Given a fundamental graded Lie algebra $\frak m=\bigoplus_{p<0} \frak g_p$, there exists a unique graded Lie algebra $\frak g(\frak m)=\bigoplus_{p\in \mathbb Z}\frak g_p(\frak m)$ such that
\begin{enumerate}
    \item[\rm(1)] $\frak g_p(\frak m)=\frak g_p$ for $p<0$
    \item[\rm(2)] if $z\in \frak g_p(\frak m)$ for $p\geq 0$ satisfies $[z, \frak m]=0$, then $z=0$.
    \item[\rm(3)] $\frak g(\frak m)$ is the largest graded Lie algebra satisfying conditions (1) and (2).
\end{enumerate}
We refer to $\frak g (\frak m)$ as the \emph{universal prolongation} of $\frak m$. Let $\frak g_0\subset \frak g_0(\frak m)$ be a subalgebra. Then, the \emph{prolongation} of $(\frak m, \frak g_0)$ is the largest graded Lie subalgebra $\frak g(\frak m, \frak g_0)=\bigoplus_p\frak g_p (\frak m, \frak g_0)$ of $\frak g(\frak m)$ such that $\bigoplus_{p<0}\frak g_{p}(\frak m, \frak g_0)=\frak m$ and $\frak g_0(\frak m, \frak g_0)=\frak g_0$.
\end{definition}

\begin{definition}\label{coboundary operator}Let $\frak m$ be a Lie algebra and $\Gamma$ be a vector space. Let $\gamma\colon \frak m \rightarrow End(\Gamma)$ be a representation of $\frak m$.

Define the \emph{coboundary operator} $\partial$ $\colon\Hom(\wedge^q\frak m, \Gamma)\rightarrow\Hom(\wedge^{q+1}\frak m, \Gamma)$ as follows: for $\phi\in\Hom(\wedge^q\frak m, \Gamma)$,
\begin{eqnarray*}
 &\partial\phi(z_1\wedge\cdots\wedge z_{q+1})=&\sum_{i=1}^{q+1}(-1)^{i+1}\gamma(z_{i})\phi(z_1\wedge \cdots \wedge \hat z_i \cdots \wedge z_{q+1}) \\
   &&+\sum_{i<j}(-1)^{i+j}\phi([z_i,z_j]\wedge z_{1}\cdots\wedge\hat{z}_{i} \cdots \wedge \hat{z}_{j} \cdots \wedge {z}_{q+1}),
 \end{eqnarray*}
 where $\hat{z}_{i}$ denotes skipping $z_i$.
 We denote the induced cochain complex by $(C(\frak m, \Gamma), \partial)$ and the derived space of the cohomology by $H(\frak m, \Gamma)$.
\end{definition}

\begin{definition} Let $\frak m=\bigoplus_{p <0}\frak g_p$ be a fundamental graded Lie algebra. Let $\Gamma =\bigoplus_{p\in \mathbb Z}\Gamma_p $ be a graded vector space and let $\gamma\colon \frak m \rightarrow End(\Gamma)$ be a representation of $\frak m$ on $\Gamma$ such that $\gamma(\frak g_{p})\Gamma_q \subset \Gamma_{p+q}$ for $p<0$ and $q \in \mathbb Z$. The cochain complex $(C(\frak m, \Gamma), \partial)$ has the following bigradation (Section 1 of \cite{Ta79} and Section 2.4 of \cite{Ya}):
 \begin{eqnarray*} C^{p,0}(\frak m, \Gamma)&=&\bigoplus_{p} \Gamma_{p-1} \\
 C^{p,q}(\frak m, \Gamma)&=&\bigoplus_{j\leq-q}\Hom(\wedge_{j}^{q}\frak m, \Gamma_{j+p+q-1}), \end{eqnarray*}
where
\begin{eqnarray*} \wedge_{j}^{q}\frak m=\bigoplus_{\scriptsize{\begin{array}{c}
                          i_1+\cdots+i_q=j\\
                          i_k<0\\
                         \end{array}}}
\frak g_{i_1}\wedge\cdots\wedge\frak g_{i_q}.
\end{eqnarray*}
Then the the coboundary operator $\partial$ maps $C^{p+1,q}(\frak m, \Gamma)$ to $C^{p,q+1}(\frak m, \Gamma)$ and the cohomology space with the bigradation
\begin{eqnarray*}H^q(\frak m, \Gamma)=\bigoplus_{p} H^{p,q}(\frak m, \Gamma)\end{eqnarray*}
is called \emph{the generalized Spencer cohomology space} of $(\Gamma, \frak m)$.
\end{definition}

Let $\frak g=\bigoplus_{p\in \mathbb Z}\frak g_p$ be a graded Lie algebra. Let $\frak m=\bigoplus_{p <0}\frak g_p$ be a graded Lie subalgebra of $\frak g$, which is fundamental. Let $\ad \colon \frak m \rightarrow End(\frak g)$ be the adjoint representation of $\frak m$ on $\frak g$ such that, for $z_1 \in \frak m$ and $z_2 \in \frak g$, $\ad(z_1)z_2=[z_1,z_2]$. The following is an effective way to show that a given graded Lie algebra $\frak g=\bigoplus_{p\in \mathbb Z}\frak g_p$ is the prolongation of $\frak m$ (or of $(\frak m, \frak g_0)$) related with the first (generalized) Spencer cohomology $H^1(\frak m, \frak g)=\bigoplus_{p} H^{p,1}(\frak m, \frak g)$.

\begin{lemma}[Lemma 2.1 of \cite{Ya}]\label{C.prolongation} Let $\frak g=\bigoplus_{p\in \mathbb Z}\frak g_p$ be a graded Lie algebra such that $\frak m=\bigoplus_{p <0}\frak g_p$ is fundamental. Then, $\frak g$ is the prolongation of $\frak m$ {\rm(resp. of $(\frak m, \frak g_0)$)} if and only if the following two conditions hold:
\begin{enumerate}
\item[\rm(1)] if $z\in \frak g_p$ for $p\geq 0$, satisfies $[z, \frak m]=0$, then $z=0$.
\item[\rm(2)] $H^{p,1}(\frak m, \frak g)=0$ for $p \geq 0\ (\text{resp. } p\geq 1)$.
\end{enumerate}
\end{lemma}

\subsection{Cartan connections}

\begin{definition}\label{filteredmanifold} Let $M$ be a manifold and $TM$ be the tangent bundle of $M$. A \emph{tangential filtration $F=\{F^p\}_{p\in\mathbb Z_{\leq 0}}$ on $M$} is a sequence of subbundles $F^p=F^pTM$ of $TM$ satisfying the following:
\begin{enumerate}
  \item $F^{p+1}\subset F^p$
  \item $F^0=0$ and $\cup F^p=TM$
  \item $[\mathcal F^p, \mathcal F^q]\subset \mathcal F^{p+q}$ for $p,q\in \mathbb Z_{\leq 0}$,
\end{enumerate}
where $\mathcal F^{\centerdot}$ is the sheaf of sections of $F^{\centerdot}$. We refer to $(M,F)$ as a \emph{filtered manifold}.

The \emph{symbol algebra $\Symb_x(F)=\bigoplus_{p\in\mathbb Z_{\leq 0}}\Symb^p_x(F)$ of $F$ at $x\in M$} is given by
\begin{eqnarray*}\Symb^p_x(F)=F^p_xTM/F^{p+1}_xTM\end{eqnarray*}
with a natural bracket induced by the Lie bracket of vector fields on $M$.

Let $\frak m=\bigoplus_{p<0} \frak g_p$ be a fundamental graded Lie algebra. A filtered manifold $(M,F)$ is called a \emph{regular filtered manifold of type $\frak m$} if the symbol algebras $\Symb_x(F)$ are all isomorphic to $\frak m$ for all $x\in M$.
\end{definition}

\begin{definition}\label{G_0 structure} Let $(M,F)$ be a regular filtered manifold of type $\frak m$. Let $\mathscr R_x(M, \frak m)$ be the set of all isomorphisms $\varsigma \colon \frak m\rightarrow \Symb_x(F)$ of graded Lie algebras. Then, with the structure group $G_{0}(\frak m)$ which consists of all automorphisms of the graded Lie algebra $\frak m=\bigoplus_{p<0}\frak g_{p}$, $\mathscr R:=\cup_{x\in M}\mathscr R_x(M, \frak m)$ is a principal $G_0(\frak m)$-bundle on $M$.
\begin{eqnarray*}
G_0(\frak m)\rightarrow &\mathscr R=\cup_{x\in M}\mathscr R_x(M, \frak m) \\
 & \downarrow \\
 & M
\end{eqnarray*}
This fiber bundle $\mathscr R$ is called the \emph{frame bundle of the regular filtered manifold $(M,F)$ of type $\frak m$}, or simply the \emph{frame bundle of $(M,F)$}.

Given a closed subgroup $G_0\subset G_{0}(\frak m)$, a \emph{$G_0$-structure} on $(M,F)$ is a $G_0$-subbundle of the frame bundle $\mathscr R$. Two $G_0$-structures on $(M_1,F_1)$ and $(M_2,F_2)$ are \emph{locally equivalent} if there exist two open subsets $U_1$ of $M_1$ and $U_2$ of $M_2$, and a $G_0$-bundle isomorphism over the open subsets $U_1$ and $U_2$.
\end{definition}

\begin{definition}
 A \emph{differential system} ($M$,$D$) on manifold $M$ is a subbundle $D$ of the tangent bundle $T(M)$ of $M$. The subbundle $D$ is \emph{completely integrable} if and only if $[D,D]\subset D$. For a non-integrable differential system $D$, we consider the \emph{derived system} $\partial D$ of $D$ which is defined in terms of sections by
\begin{eqnarray*}\partial \mathcal D=\mathcal D+[\mathcal D, \mathcal D],\end{eqnarray*}
 where $\mathcal D$ denotes the space of sections of $D$.
Moreover, \emph{the k-th weak derived systems} $\partial^{(k)} D$ of $D$ are inductively defined by
\begin{eqnarray*}\partial^{(k)}\mathcal D=\partial^{(k-1)} \mathcal D+[\mathcal D, \partial^{(k-1)}\mathcal D], \end{eqnarray*}
where $\partial^{(0)} D=D$ and $\partial^{(k)} \mathcal D$ denotes the space of sections of $\partial^{(k)}D$. A differential system ($M$,$D$) is said to be {\it regular} if $D^{-(k+1)}:=\partial^{(k)}D$ is a subbundle of $T(M)$ for every integer $k\geq1$.
 For a regular differential system ($M$,$D$) such that $D^{-\mu}=T(M)$, we define the associated graded algebra $\frak m(x)$ at $x\in M$, which was introduced by Noboru Tanaka in \cite{Ta70}. We put $\frak g_{-1}(x)=D^{-1}(x)$, $\frak g_{p}(x)=D^{p}(x)/D^{p+1}(x)$ (for $p<-1$) and
 \begin{eqnarray*}\frak m(x)=\bigoplus_{p=-1}^{-\mu}\frak g_{p}(x).\end{eqnarray*}
 Then, $\frak m(x)$ becomes a fundamental graded Lie algebra, which we refer to as \emph{the symbol algebra} of ($M$,$D$) at $x\in M$. If the symbol algebra $\frak m(x)$ is isomorphic to a given fundamental graded Lie algebra for each $x\in M$, then we refer to ($M$,$D$) as a {\it regular differential system of type $\frak m$}.
\end{definition}

\begin{definition} Let $\frak m=\bigoplus_{p<0} \frak g_p$ be a fundamental graded Lie algebra. A regular tangential filtration $(M,F)$ of type $\frak m$ derived from a regular differential system $(M,D)$ of type $\frak m$ is $F^p=D^p$ for $p \leq 0$.
We just denote $(M,D)$ as a \emph{regular filtered manifold derived from a regular differential system} $(M,D)$ of type $\frak m$. A \emph{$G_0$-structure on $(M,D)$} is a $G_0$-subbundle of the frame bundle $\mathscr R$ of $(M,D)$.
\end{definition}

\begin{definition} Let $\frak g$ be a Lie algebra and let $\frak h\subset \frak g$ be a subalgebra. Let $H$ be a connected Lie group with Lie algebra $\frak h$ and let $\Ad\colon H\rightarrow GL(\frak g)$ be the adjoint representation of $H$ on $\frak g$. A \emph{Cartan connection of type $(\frak g, H)$} on a manifold $M$ is a principal $H$-bundle $\pi \colon P\rightarrow M$ with $\frak g$-valued 1-form $\omega$ on $P$ such that
\begin{enumerate}
  \item[\rm(1)] $\omega(z^{\dagger})=z$ for $z\in \frak h$ where $z^\dagger$ denotes the fundamental vector field on $P$ induced by $z\in \frak h$;
  \item[\rm(2)] $R^*_h \omega=\Ad(h^{-1})\omega$ for $h\in H$ where $R_h \colon P\rightarrow P$ is the right action of $h\in H$ on $P$;
  \item[\rm(3)] the linear map $\omega_p \colon T_p(P)\rightarrow \frak g$ is a vector space isomorphism for each $p\in P$.
\end{enumerate}

Two Cartan connections of type $(\frak g, H)$, denoted by pairs $(P_1, \omega_1)$ on $M_1$ and $(P_2, \omega_2)$ on $M_2$, are \emph{locally equivalent} if there exist two open subsets $U_1$ of $M_1$ and $U_2$ of $M_2$, and a biholomorphic map $\phi \colon P_1|_{U_1}\rightarrow P_2|_{U_2}$ descending to $U_1\rightarrow U_2$ such that $\phi^* \omega_2=\omega_1$. A Cartan connection of type $(\frak g, H)$ is \emph{locally flat} if it is locally equivalent to the Cartan connection on the principal $H$-bundle $G\rightarrow G/H$ with the Maurer-Cartan form on $G$ where $G$ is a connected Lie group with Lie algebra $\frak g$ and an inclusion $H\subset G$ as a closed subgroup.
\end{definition}

Let $V$ and $W$ be vector spaces with filtration. Then
 \begin{eqnarray*}
 F^p\Hom(V,W)&=&\{\alpha \in \Hom(V,W) \ | \ \alpha(F^i V)\subset F^{i+p}W \ \forall i  \}\\
 F^p GL(V)&=&\{\alpha \in GL(V)\ |\ \alpha-1_V \in F^p Hom(V,V) \}\\
 F^p \Aut(V)&=&\Aut(V)\cap F^p GL(V)
 \end{eqnarray*}

\begin{definition}Let $\frak m=\bigoplus_{p<0}\frak g_{p}$ be a fundamental graded Lie algebra. Let $H$ be a Lie group and $E=\frak m \oplus \frak h$ where $\frak h$ is the Lie algebra of $H$. An \emph{skeleton} on $\frak m$ is a pair $(E, H)$ with a representation $\rho$ of $H$ on $E$ satisfying the followings:
\begin{enumerate}
  \item $\rho(h)z=Ad(h)z$ for $h\in H$, $z \in \frak h$.
  \item $\rho(h)F^p\frak m \subset F^p\frak m\oplus \frak h$ for $h \in H$ and $p<0$, where $F^p\frak m=\bigoplus_{p\leq i \leq -1}\frak g_{i}$
  \item There exist an filtration $\{F^p H\}$ on $H$, and hence there is induced filtration $F^{p}E=F^p\frak m \oplus F^p \frak h$ on $E$, where $F^p \frak h$ is the Lie algebra of $F^p H$ satisfying: for $p\leq 0$, $F^p H=H$ and, for $p\geq 0$, the sequences
\begin{eqnarray*}                                 1\longrightarrow F^{p+1}H \longrightarrow H \stackrel{\rho^{p}}{\longrightarrow} F^{0}Aut(E^{(p-1)})/F^{p+1}Aut(E^{(p-1)})
\end{eqnarray*}
are exact, where $E^{(p-1)}=E/F^{p}E$, and $\rho^{p}$ is the homomorphism induced by $\rho$.
\end{enumerate}
\end{definition}

\begin{definition}
Let $(M,F)$ be a regular filtered manifold of type $\frak m$ and $(E,H)$ be a skeleton on $\frak m$. A \emph{tower $P$} on $M$ with skeleton $(E,H)$ is a principal $H$-bundle $\pi: P \rightarrow M$ with an $E$-valued 1 form $\theta$ satisfying:
\begin{enumerate}
  \item[\rm(1)] the linear map $\theta_p \colon T_p(P)\rightarrow E$ is a filtered vector space isomorphism for each $p\in P$;
  \item[\rm(2)] $\theta(z^{\dagger})=z$ for $z\in \frak h$ where $z^\dagger$ denotes the fundamental vector field on $P$ induced by $z\in \frak h$;
  \item[\rm(3)] $R^*_h \theta=\Ad(h^{-1})\theta$ for $h\in H$ where $R_h \colon P\rightarrow P$ is the right action of $h\in H$ on $P$.
\end{enumerate}
\end{definition}

Let $\frak m$ be a fundamental graded Lie algebra. Let $G_0$ be a connected Lie subgroup of $G_0(\frak m)$ and let $\frak g_0$ be the Lie algebra corresponding to $G_0$. Let $\frak g(\frak m, \frak g_0)$ be the prolongation of the graded Lie algebra of $(\frak m, \frak g_0)$. Let $\frak M$ be a Lie group having $\frak m$ as its Lie algebra. The trivial subbundle $\frak M \times G_0$ of the frame bundle $\frak M \times G_0(\frak m)$ is the standard $G_0$-structure on $(\frak M, \frak m)$.

By Theorem 2.3.2 and Theorem 3.6.1 of [16], we can construct a tower $P$ on $\frak M$ with skeleton $(\frak g(\frak m, \frak g_0),H:=H(\frak m, G_0))$. More precisely, by Theorem 2.3.2 of \cite{Mo}, for a given $P^{(0)}:=\frak M \times G_0$, we can construct a unique tower $\mathcal RP^{(0)}$ satisfying the universal property: $\mathcal RP^{(0)}/F^{1}\mathcal RP^{(0)}= P^{(0)}$ and any tower $Q$ on $\frak M$ with $Q/F^{1}Q\subset P^{(0)}$ is embedded in the tower $\mathcal RP^{(0)}$. The tower $\mathcal RP^{(0)}$ is called \emph{universal tower prolonging $P^{(0)}$}. In the proof of Theorem 3.6.1 of \cite{Mo}, we get a tower $P^{(1)}$ and a surjective map $P^{(1)} \rightarrow \mathcal RP^{(0)}/F^{2}$ where $F^{2}=F^{2}RP^{(0)}$. Apply these two theorem again to obtain universal tower $\mathcal RP^{(1)}$ prolonging $P^{(1)}$, and a tower $P^{(2)}$ with a surjective map $P^{(2)} \rightarrow \mathcal RP^{(1)}/F^{3}$, where $F^{3}=F^{3}RP^{(1)}$. In this way, starting from $P^{(0)}:=\frak M \times G_0$, we get a tower $P$ which is the limit of the sequence of the bundles ($P^{(0)}$, $P^{(1)}$, $P^{(2)}$, $\cdots$).
\begin{displaymath}
\xymatrix{
\cdots& P^{(2)} \ar[dr] & \ & P^{(1)}\ar[dr] & \ & P^{(0)}\\
\mathcal RP^{(2)}/F^{4}\ar[ur] \ar[dr]& \  & \mathcal RP^{(1)}/F^{3} \ar[ur] \ar[dr] & \ & \mathcal RP^{(0)}/F^{2} \ar[ur] \ar[dr]\\
\cdots &0 \ar[ur] &  &0 \ar[ur] &  & 0
}
\end{displaymath}
are exact.
We denote the structure group of $P$ by
\begin{eqnarray}\label{H} H(\frak m, G_0).\end{eqnarray}

Let $\frak h(\frak m, \frak g_0)$ be the Lie algebra of the structure group $H(\frak m, G_0)$ of $P$. By the construction of $P$, we see $\frak h(\frak m, \frak g_0)=\bigoplus_{p\geq 0}\frak g_p(\frak m, \frak g_0)$, so that $\frak g(\frak m, \frak g_0)=\frak m \oplus \frak h(\frak m, \frak g_0)$. The group $G_0$ is embedded in $H(\frak m, G_0)$ as a closed subgroup.

We define a subspace of $\Hom(\wedge^2 \frak m, \frak g (\frak m, \frak g_0))$ as
\begin{eqnarray*}
F^1\Hom(\wedge^2 \frak m, \frak g (\frak m, \frak g_0)):=\{\alpha  | \alpha(\frak g_i \wedge \frak g_j)\subset \bigoplus_{p \geq i+j+1}\frak g (\frak m, \frak g_0)_{p} \text{ for } i, j <0 \}.
\end{eqnarray*}

There are general conditions for the existence of Cartan connections. Let $(M,F)$ be a regular filtered manifold of type $\frak m$, and let $R^{(0)}$ be a $G_0$-structure on $(M,F)$. Under the assumption of following Theorem, we can construct a principal $H(\frak m, G_0)$-bundle $R\rightarrow M$, which is obtained by extending the first order frame bundle $R^{(0)}$. The principal $H(\frak m, G_0)$-bundle $R\rightarrow M$ is a tower with skeleton $(\frak g(\frak g_0,\frak m),H(\frak m, G_0))$ with $\frak g(\frak g_0,\frak m)$-valued 1-form $\omega$. Then $(P, \omega)$ is a Cartan connection of type $(\frak g (\frak m, \frak g_0), H(\frak m,G_0))$. For further details, see Chapters 2 and 3 of \cite{Mo} and Theorem 2.7 of \cite{Ta79}.

\begin{theorem}[Definition 3.10.1 and Theorem 3.10.1 of \cite{Mo}, and Theorem 2.7 of \cite{Ta79}]\label{C.cartan connection}
 Let $(M,F)$ be a regular filtered manifold of type $\frak m$, and let $G_0$ be a Lie subgroup of $G_0(\frak m)$ with Lie algebra $\frak g_0$.
 Suppose that there exists a subspace $W$ of $F^1\Hom(\wedge^2 \frak m, \frak g (\frak m, \frak g_0))$ such that
\begin{enumerate}
  \item[\rm(1)] $F^1\Hom(\wedge^2 \frak m, \frak g (\frak m, \frak g_0))=W \oplus \partial F^1 \Hom(\frak m, \frak g (\frak m, \frak g_0))$,
  \item[\rm(2)] $W \text{ is stable under the action of } H:=H(\frak m, G_0)$.
\end{enumerate}
Then, for each $G_0$-structure on $(M,F)$, we can construct a principal $H$-bundle $P\rightarrow M$ associated with the $G_0$-structure on $(M,F)$ and obtain a Cartan connection $(P, \omega)$ of type $(\frak g (\frak m, \frak g_0), H)$. Two $G_0$-structures on $(M,F)$ are (locally) equivalent when the associated Cartan connections of type $(\frak g (\frak m, \frak g_0), H)$ are (locally) equivalent.
\end{theorem}

\begin{theorem}[Proposition 3.10.1 of \cite{Mo}]\label{Doubrov} Let $\frak m$ be a fundamental graded Lie algebra. Let $G_0$ be a Lie subgroup of $G_0(\frak m)$. Let $\frak g_0$ be the subalgebra of $\frak g_0(\frak m)$ corresponding to $G_0$, $\frak g =\frak g(\frak m, \frak g_0)$ be the prolongation of $(\frak m, \frak g_0)$, and $\frak h = \bigoplus_{p \geq 0} \frak g_p$ be its non-negative part. Assume that the prolongation $\frak g$ is finite-dimensional and that there exist a positive definite bilinear form
\begin{eqnarray*}( , ) : \frak g \times \frak g \to \mathbb R, \end{eqnarray*}
a mapping $\tau \colon \frak h \to \frak g$ and a mapping $\tau_0 \colon G_0 \to G_0$ such that
\begin{enumerate}
  \item[\rm(1)] $(\frak g_p, \frak g_q) = 0$ for $p \neq q$
  \item[\rm(2)] $\tau(\frak g_p)\subset \frak g_{-p}$ for $p \geq 0$, and $([A,z_1],z_2)=(z_1,[\tau(A),z_2])$ for all $z_1, z_2 \in \frak g$ and $A \in \frak h$
  \item[\rm(3)] $(az_1,z_2)=(z_1,\tau_0(a)z_2)$ for $z_1, z_2 \in \frak g$ and $a \in G_0$
\end{enumerate}
Then, there exists a full functor from the category of $G_0$-structures of type $\frak m$ to the category of Cartan connections of type $(\frak g,H)$, where $H$ is the Lie group $H(\frak m, G_0)$ with its Lie algebra $\frak h$.
\end{theorem}

The following theorem is essentially from Proposition 3.10.1 of \cite{Mo} and will be applied to the Lie algebras $\frak g$ of the automorphism groups of nonhomogeneous smooth horospherical varieties of Picard number one.

\begin{theorem}\label{connection} Let $\frak m$ be a fundamental graded Lie algebra. Let $G_0$ be a connected Lie subgroup of $G_0(\frak m)$ and let $\frak g_0$ be the Lie algebra of $G_0$. Let $\frak g(\frak m, \frak g_0)=\bigoplus_{p\in \mathbb Z}\frak g_p$ be the prolongation of $(\frak m, \frak g_0)$ and $\frak h(\frak m, \frak g_0)=\bigoplus_{p \geq 0} \frak g_p$ be its non-negative part. Let $H:=H(\frak m, G_0)$ be a Lie group given as {\rm (\ref{H})} with its Lie algebra $\frak h(\frak m, \frak g_0)$.

Assume that the prolongation $\frak g(\frak m, \frak g_0)$ is finite-dimensional. We also assume that there exist a graded Lie algebra $\tilde{\frak g}$ that contains $\frak g(\frak m, \frak g_0)$ as a Lie subalgebra, an $\ad(\frak g(\frak m, \frak g_0))$-invariant symmetric bilinear form $(.,.)$ on $\tilde{\frak g}$, and a map $\tau \colon \tilde{\frak g} \rightarrow \tilde{\frak g}$ satisfying
    \begin{enumerate}
      \item[\rm(1)] $\{.,.\}:=-(., \tau.)$ is a positive definite Hermitian inner product on $\frak g(\frak m, \frak g_0)$ and $\{\frak g_p,\frak g_q \}=0$ if $p \neq q$;
      \item[\rm(2)] $\frak g_{p} \subset \tilde{\frak g}_{p}$ for any integer $p$ and $\tau(\frak g_p)\subset \tilde{\frak g}_{-p}$ for any integer $p \geq 0$;
      \item[\rm(3)] $\{[A,z_1],z_2\}=-\{z_1,[\tau(A),z_2]\}$ for $A \in \frak g$, $z_1, z_2\in \tilde{\frak g}$;
      \item[\rm(4)] there exists $\tau_0 \colon G_0 \rightarrow \tilde{G_0}$ such that and $\{az_1,z_2\}=-\{z_1, \tau_0(a)z_2\}$ for $z_1, z_2\in \tilde{\frak g}$ and $a\in G_0$, where $\tilde{G_0}$ is a closed subgroup of $G_0(\tilde {\frak m}):=\Aut(\bigoplus_{p<0}\tilde{\frak g}_{p})$ with its Lie algebra $\tilde{\frak g}_0$.
    \end{enumerate}
Then, for each $G_0$-structure on $(M,F)$ of type $\frak m$, we can construct a Cartan connection $(P, \omega)$ of type $(\frak g(\frak m, \frak g_0), H)$ so that two $G_0$-structures on $(M,F)$ are (locally) equivalent when the associated Cartan connections of type $(\frak g(\frak m, \frak g_0), H)$ are (locally) equivalent. \end{theorem}

\begin{proof}
We simplify $\frak g=\frak g (\frak m, \frak g_0)$, $\frak h=\frak h (\frak m, \frak g_0)$. By Theorem \ref{C.cartan connection}, it is sufficient to show that there exists a subspace $W$ of $F^1\Hom(\wedge^2 \frak m, \frak g)$ such that
\begin{enumerate}
  \item $F^1\Hom(\wedge^2 \frak m, \frak g)=W \oplus \partial F^1 \Hom(\frak m, \frak g)$,
  \item $W \text{ is stable under the action of } H$.
\end{enumerate}

Let $\frak m'$ be the dual of $\frak m$. We extend the bilinear form $(\cdot,\cdot)$ to $\wedge^\centerdot \tilde{\frak g}$. We identify $\wedge^\centerdot \frak m'$ and $\wedge^\centerdot \tau (\frak m)$ by defining a map $\eta\colon \wedge^\centerdot \frak m'\rightarrow \wedge^\centerdot \tau (\frak m)\subset \wedge^\centerdot \tilde{\frak g}$ as the inverse of the isomorphism
\begin{eqnarray*}
\eta'\colon \wedge^\centerdot \tau (\frak m) &\rightarrow& \wedge^\centerdot \frak m' \\
         a &\mapsto& \eta'(a): z \mapsto (a,z) \text{ for } z \in \wedge^\centerdot \frak m.
\end{eqnarray*}
Then, $(\eta(f),z)=f(z)$ for $f\in \wedge^\centerdot \frak m'$ and $z\in \wedge^\centerdot \frak m$.

We also extend the bilinear form $(\cdot,\cdot)$ to $\Hom(\wedge^{\cdot}\tilde{\frak m}, \tilde{\frak g})$ and define $\{.,.\}:=-(., \tau.)$. By assumption (1), the extended Hermitian inner product $\{.,.\}$ is positive definite on $\Hom(\wedge^{\centerdot} \frak m, \frak g) \cong \wedge^\centerdot \tau (\frak m)\otimes \frak g$. Let $\partial^*$ be the formal adjoint of $\partial$ with respect to the extended Hermitian inner product $\{.,.\}$ on $\Hom(\wedge^{\centerdot}\frak m, \frak g)$.

Then, we have the direct sum decomposition
\begin{eqnarray*} \Hom(\wedge^q \frak m, \frak g)=\partial \Hom(\wedge^{q-1} \frak m, \frak g)\oplus \Ker \partial^*.\end{eqnarray*}

Let us show that $\Ker\partial^*$ is an invariant subspace for the action of $H$.
Let $\rho$ be the representation of $H\subset G$ on $\Hom(\wedge^{\centerdot}\frak m, \frak g)$ and $\rho_*$ be the corresponding adjoint representation of $\frak h\subset \frak g$ on $\Hom(\wedge^{\centerdot}\frak m, \frak g)$. Since any element $a\in H$ is written as
\begin{eqnarray*} a=a_0 \cdot \exp(A)
\end{eqnarray*}
with $a_0\in G_0$, $A\in F^1\frak h:= \bigoplus_{p>0}\frak g_p$,
it suffices to show that
\begin{eqnarray*}
&\mbox{(a)}& \partial^* \circ \rho (a_0)=\rho (a_0) \circ \partial^*  \text{ for } a_0\in G_0 \\
&\mbox{(b)}& \partial^* \circ \rho_*(A)=\rho_* (A) \circ \partial^* \text{ for } A \in F^1\frak h.
\end{eqnarray*}

Let $\tilde\partial$ be the coboundary operator on $\Hom(\wedge^{\cdot}\tilde{\frak m}, \tilde{\frak g})$. Let $\tilde\rho$ be the representation of $\tilde{G_0}$ on $\Hom(\wedge^{\cdot}\tilde{\frak m}, \tilde{\frak g})$ and $\tilde{\lambda}$ be the adjoint representation of $\tilde{\frak m} = \bigoplus_{p < 0}\tilde{\frak g}_p$ on $\Hom(\wedge^{\cdot}\tilde{\frak m}, \tilde{\frak g})$.

In general, we have
\begin{eqnarray*}
\tilde \partial \circ \tilde\rho (b_0)=\tilde \rho (b_0) \circ \tilde \partial  \text{ for } b_0 \in \tilde{G_0},
\end{eqnarray*}
and
\begin{eqnarray}\label{c1}
\tilde\partial \circ \tilde\lambda(B)=\tilde \lambda (B) \circ \tilde\partial \text{ for } B\in \tilde{\frak m}.\end{eqnarray}

We denote $\tilde{\frak g}=\frak g \oplus \frak g^\perp$, where $\frak g^\perp = \{z \in \tilde{\frak g} | \{z, y\}=0 \text{ for all } y \in \frak g \}$. Let $\lambda$ be the representation of $\tau(F^1\frak h)$ on $\Hom(\wedge^{\centerdot} \frak m, \frak g) \cong \wedge^\centerdot \tau (\frak m)\otimes \frak g$, which is defined by a composition of the adjoint representation of $\tau(F^1\frak h)$ on $\wedge^{\centerdot}\tau(\frak m)\otimes \tilde{\frak g}$ and the projection $\pi \colon \frak g \oplus \frak g^\perp \rightarrow \frak g$:
\begin{eqnarray*} \lambda=(id \otimes \pi)\circ (\ad \otimes \ad)\colon \tau(F^1\frak h) \rightarrow \End(\wedge^{\centerdot}\tau(\frak m)\otimes {\frak g}).\end{eqnarray*}

It follows that for $A\in \frak h$ and $\phi, \psi\in \Hom(\wedge^{\centerdot}\frak m, \frak g)$,
\begin{eqnarray*}
\{ \phi, \tilde\partial \tilde\lambda(\tau A)\psi \}&=&\{ \phi, \partial \tilde\lambda(\tau A)\psi \} \\
                                                    &=&\{\partial^* \phi, \tilde\lambda(\tau A)\psi \} \\
                                                    &=&\{ \partial^* \phi, \lambda(\tau A)\psi \} \\
                                                    &=&\{\phi, \partial \lambda(\tau A)\psi \}
\end{eqnarray*}
and hence,
\begin{eqnarray*}
\{\phi, \partial \lambda(\tau A)\psi \}
&=&\{\phi, \tilde\partial \tilde\lambda(\tau A)\psi \}\\
&=&\{\phi, \tilde\lambda(\tau A)\tilde\partial \psi \}\text{ because (\ref{c1}) and $\tau(F^1\frak h) \subset \tilde{\frak m}$ by assumption (2)} \\
&=&\{\phi, \tilde\lambda(\tau A)\partial \psi \}\\
&=&\{\phi, \lambda(\tau A)\partial \psi \}.
\end{eqnarray*}
Thus,
\begin{eqnarray}\label{c2}
\partial \circ \lambda(B)= \lambda (B) \circ \partial \text{ for } B\in \tau(F^1\frak h).
\end{eqnarray}

Hence,
\begin{eqnarray*}
\{\partial^* \circ \rho_*(A) \phi, \psi \}
&=&\{\rho_*(A) \phi, \partial \psi \} \\
&=&-\{\phi, \lambda(\tau A)\partial \psi \} \\
&=&-\{\phi, \partial \lambda(\tau A)\psi \} \text{ from (\ref{c2}) } \\
&=&-\{\partial^* \phi, \lambda(\tau A)\psi \}\\
&=&\{\rho_*(A)\partial^*\phi, \psi\},
\end{eqnarray*}
which gives (b).

Similarly, we can verify (a).

If we set $W=\Ker\partial^*\cap F^1\Hom(\wedge^2\frak m, \frak g)$, the proof is completed by Theorem \ref{C.cartan connection}.
\end{proof}

\section{Lie algebras of the automorphism groups of horospherical varieties}\label{section2}

Horospherical varieties are complex normal algebraic varieties where a connected reductive algebraic group $L$ acts with an open orbit isomorphic to a torus bundle over a flag variety (\cite{Pa}).

\begin{theorem}[Theorem 0.1 and Theorem 1.11 of \cite{Pa}]\label{Horospheical}
 Let $X$ be a smooth nonhomogeneous projective horospherical $L$-variety with Picard number one. Then, the automorphism group of $X$ is a connected non-reductive linear algebraic group $G$, acting with exactly two orbits. Moreover, $X$ is uniquely determined by its two closed $L$-orbits $Y$ and $Z$, which are isomorphic to $L/P_\alpha$ and $L/P_\beta$, respectively. Let $\pi_i$ be the $i$-th fundamental weight of $L$-representation space. The variety $X=(L, \alpha, \beta)$ is one of the triples, with the group $G$, of the following list.
 \begin{enumerate}
   \item[\rm(1)] $(B_m,{\alpha_{m-1}},{\alpha_{m}})$ for $m\geq3$ and $(\SO(2m+1)\times \mathbb C^*)\ltimes V(\pi_m)$
   \item[\rm(2)] $(B_3,{\alpha_{1}},{\alpha_{3}})$ and $(\SO(7)\times \mathbb C^*)\ltimes V(\pi_3)$
   \item[\rm(3)] $(C_m,{\alpha_{i}},{\alpha_{i+1}})$ for $m\geq2$, $i\in \{1,\ldots, m-1\}$ and $((\Sp(2m)\times \mathbb C^*)/\{\pm1\})\ltimes V(\pi_1)$
   \item[\rm(4)] $(F_4,{\alpha_{2}},{\alpha_3})$ where $\alpha_{2}$ is a long root and $(F_4\times \mathbb C^*)\ltimes V(\pi_4)$
   \item[\rm(5)] $(G_2,{\alpha_{2}},{\alpha_1})$ and $(G_2\times \mathbb C^*)\ltimes V(\pi_1)$
 \end{enumerate}
  Here, $P_{\alpha_i}$ is the maximal parabolic subgroup of $L$ associated with the simple root $\alpha_i$, and $V(\pi_i)$ is the irreducible $L$-representation with the highest weight $\pi_i$.
\end{theorem}

 For a given $X=(L, \alpha, \beta)$, there are irreducible $L$-representations $V(\pi_\alpha)$ and $V(\pi_\beta)$, and the highest weight vectors $v_{\alpha}$ of $V(\pi_\alpha)$ and $v_{\beta}$ of $V(\pi_\beta)$ such that $X$ is the orbit closure of $L.[v_{\alpha}+v_{\beta}]\subset \mathbb P (V(\pi_{\alpha})+V(\pi_{\beta}))$ (Section 1.3 of \cite{Pa}). Hence, $X$ has three orbits under the action of $L$: one open orbit isomorphic to a torus bundle over $L/(P_\alpha\cap P_\beta)$, and two closed orbits $Y$ and $Z$ that are isomorphic to $L/P_\alpha$ and $L/P_\beta$, respectively.

 Let $G$ be the automorphism group of $X$. According to Lemma 1.15 of \cite{Pa}, the closed $L$-orbit $Z$ is stable under the $G$-action. Let $\tilde{X}$ be the blowup of $X$ along $Z$. Then, $G=\Aut\tilde{X}$. According to the proof of Lemma 1.17 of \cite{Pa}, $\tilde{X}$ is a projective bundle over the $L$-orbit $Y$ and $U\subset G$ acts on $\tilde{X}$ by translation on the fibers of $\tilde{X}\rightarrow Y \cong L/P_\alpha$. Further, $G=(L\times \mathbb C^*)/C\ltimes U$, where $U$ is an $L$-representation space and $C$ is the centralizer.

\begin{proposition}\label{imbedding}
Let $X=(L, \alpha, \beta)$ be a smooth nonhomogeneous projective horospherical variety of Picard number one. Let $\frak g$ be the Lie algebra of the automorphism group of $X$. Then,
\begin{enumerate}
  \item[\rm(1)] the Lie algebra $\frak g$ is a semidirect product of $(\frak l+\mathbb C)$ and an irreducible $\frak l$-representation $U$, where $\frak l$ is a semisimple Lie algebra, i.e., $\frak g=(\frak l+\mathbb C) \rhd U$;
  \item[\rm(2)] there exist two irreducible $L$-representations $V_{\alpha}$ and $V_{\beta}$ such that $\frak l\subset \End(V_{\alpha})$, $\frak l\subset \End(V_{\beta})$,  $\mathbb C \simeq \mathbb C I\subset \End(V_{\beta})$, and $U\subset \End(V_{\alpha},V_{\beta})$. Hence, we regard $\frak g$ as a Lie subalgebra via the inclusion $i \colon \frak g \hookrightarrow  \frak{gl}(V)=\End V$ where $V=V_{\alpha}\oplus V_{\beta}$. In particular, we can write an element $\bf{Z}$ of $\frak g$ as
   \begin{eqnarray*}
   \mathbf{Z}=\left(\begin {array}{cc}
   l & 0 \\
   u & l+ c
   \end{array} \right)\in \End(V)=\frak{gl}(V)
   \end{eqnarray*}
   where $l\in \frak l$, $u\in U$, and $c\in \mathbb C I$.
 \end{enumerate}

Let $*$ be the operator on $\frak{gl}(V)$ given by $z^*=\bar{z} ^{t}$ for $z\in \frak{gl}(V)$. Let $\tau$ be an operator defined by $\tau(z)=-z^*$ for $z\in \frak {gl}(V)$. Let $(.,.)$ be the Cartan-Killing form on $\frak{gl}(V)$.
\begin{enumerate}
  \item[\rm(3)] We define an inner product $\{\cdot, \cdot \}$ by $\{z_1,z_2\}=(z_1,z_2^*)=-(z_1, \tau(z_2))$ for $z_1, z_2\in\frak{gl}(V)$. Then, a restricted inner product $\{\cdot, \cdot \}$ is a positive definite Hermitian inner product on $\frak g$.
\end{enumerate}
\end{proposition}

\begin{proof}
\begin{enumerate}
\item It is from Theorem 1.11 of \cite{Pa}.
\item It is from the proof of Theorem 1.1. of \cite{Pa}. Since $X$ is the orbit closure of $L.[v_{\alpha}+v_{\beta}]\subset \mathbb P (V(\pi_{\alpha})+V(\pi_{\beta}))$, let $V_{\alpha}=V(\pi_{\alpha})$ and $V_{\beta}=V(\pi_{\beta})$.
\item If we take two elements $\mathbf{Z_1}$ and $\mathbf{Z_2}$ in $\frak g$,
  \begin{eqnarray*}
   \mathbf{Z_1}=\left(\begin {array}{cc}
   l_1 & 0 \\
   u_1& l_1+c_1
   \end{array} \right)\text{ and }
   \mathbf{Z_2}=\left(\begin {array}{cc}
   l_2 & 0 \\
   u_2 & l_2+c_2
   \end{array} \right)
  \end{eqnarray*}
where $l_1, l_2\in \frak l$, $u_1, u_2\in U$ and $c_1, c_2 \in \mathbb C$.
Then,
 \begin{eqnarray*}
 \mathbf{Z_1}\mathbf{Z_2^*}=\left(\begin {array}{cc}
   l_1 l_2^* & l_1 u_2^* \\
   u_1 l_2^* & u_1 u_2^*+l_1 c_2^*+c_1 l_2^*+c_1 c_2^*
   \end{array} \right).
 \end{eqnarray*}

From page 271 of \cite{Jo}, we see that
  \begin{eqnarray*}
 Tr \ad \mathbf{X} \ad \mathbf{Y} = 2n Tr(\mathbf{X}\mathbf{Y})- 2Tr(\mathbf{X})Tr(\mathbf{Y})\end{eqnarray*}
for $\mathbf{X}, \mathbf{Y} \in \frak{gl}(V)$.

   Since the semisimple Lie algebra $\frak l$ in $\frak{gl}(V)$ is contained in $\frak{sl}(V)$ which is the traceless subalgebra of $\frak{gl}(V)$,
  \begin{eqnarray*}
 \{\mathbf{Z_1}, \mathbf{Z_2}\}&=&2n Tr(\mathbf{Z_1}\mathbf{Z_2}^*)- 2Tr(\mathbf{Z_1})Tr(\mathbf{Z_2}^*) \\
                         &=&2n Tr(l_1 {l_2}^*)+2nTr( u_1 u_2^*)+2n_{\alpha}n_{\beta}c_1 \cdot c_2^*
 \end{eqnarray*}
 where $n=\dim(V)$, $n_{\alpha}=\dim(V_{\alpha})$, and $n_{\beta}=\dim(V_{\beta})$.
    Hence, $\{\cdot, \cdot \}$ is a positive definite Hermitian inner product on $\frak g$.
\end{enumerate}
\end{proof}

\begin{remark} We rescale the Hermitian inner product on $\frak g$ via division by $2n$ for $n=\dim(V)$(respectively, rescale the Cartan-Killing form). Thus,
 \begin{eqnarray*}
 \{\mathbf{Z_1}, \mathbf{Z_2}\}&=& Tr(l_1 {l_2}^*)+Tr( u_1 u_2^*)+\frac{n_{\alpha}n_{\beta}}{n}c_1 \cdot c_2^*.
 \end{eqnarray*}
  Then, for $E_{ij} \in V^*_{\alpha}\otimes V_{\beta}$ which is zero except $ij$-component or if we write a unit column vector $e_i$ in the $j$-th entry, we see $\{E_{ij},E_{kl}\}=Tr(E_{ij},E^*_{kl})=\delta_{jl}e_i \cdot e^*_k=\delta_{ik}\delta_{jl}$.
\end{remark}

Let $\frak l$ be a semisimple Lie algebra with $\rank(\frak l)=m$. We fix a Cartan subalgebra $\frak h$. Let $\Phi$ be a set of roots of $\frak l$ relative to $\frak h$. The root space decomposition of $\frak g$ is
 \begin{eqnarray*} \frak l =\frak h \oplus \bigoplus_{\alpha \in \Phi} \frak l_{\alpha}, \end{eqnarray*}
where $\frak l_{\alpha} =\{X \in \frak g \ | \ [ H, X ] = \alpha(H)X \text{ for all } H \in \frak h  \}$ is the root space for $\alpha \in \Phi$.

\begin{definition}
Let $\triangle=\{\alpha_1, \cdots, \alpha_m \}$ be a set of simple roots of $\frak l$ associated with the Cartan subalgebra $\frak h$. We define the \emph{characteristic element $E_{\alpha_i}$ associated with $\alpha_i\in\triangle$} as
\begin{eqnarray*}
  \alpha_j(E_{\alpha_i})=\left\{ \begin{array}{cc} &1\ \ \text{ if } \ j=i \\
                                     &0\ \ \text{ if } \ j \neq i,  \end{array} \right.
\end{eqnarray*}
Then, we can construct a gradation $\frak l=\bigoplus_{p\in \mathbb Z} \frak l_p$ which is called a \emph{gradation associated with $E_{\alpha_i}$} as follows:
\begin{eqnarray*}
&&\frak l_0 = \frak h \oplus \bigoplus_{\alpha \in \Phi_0^{+}} \frak l_{\alpha} \oplus \frak l_{-\alpha}\\
&&\frak l_k =\bigoplus_{\alpha \in \Phi_k^{+}} \frak l_{\alpha} \\
&&\frak l_{-k} =\bigoplus_{\alpha \in \Phi_k^{+}} \frak l_{-\alpha} \ (k>0),
\end{eqnarray*}
where $\Phi_k^{+}=\{ \alpha \in \Phi^{+} | \alpha(E_{\alpha_i})=k \}$.

Then, we can construct a gradation $\frak l=\bigoplus_{p\in \mathbb Z} \frak l_p$ which is called a \emph{gradation associated with $\alpha_i$}. In this
case, by Lemma 3.8 of \cite{Ya}, $\bigoplus_{p<0}\frak l_p$ is a fundamental graded Lie algebra.
\end{definition}

\begin{example}
Let $L$ be a semisimple Lie group. Let $P_{\alpha_i}$ be a maximal parabolic subgroup of $L$ associated with a simple root $\alpha_i$. The Lie algebra $\frak l$ of $L$ has a gradation $\bigoplus_{p\in \mathbb Z} \frak l_p$ associated with $\alpha_i$. Then, the tangent space of the homogeneous space $L/P_{\alpha_i}$ at each point is identified with $\bigoplus_{p<0}\frak l_p$ which is a fundamental graded Lie algebra.
\end{example}

\begin{proposition}\label{tangentsp of X}
Let $X$ be a smooth nonhomogeneous projective horospherical variety $(L, \alpha, \beta)$ of Picard number one. Let $G=\Aut(X)$ and let $\frak g=(\frak l+\mathbb C) \rhd U$ be the corresponding Lie algebra. Then, we can give a gradation of $\frak g=\bigoplus_p \frak g_p$ such that the graded Lie algebra $\frak m=\bigoplus_{p<0}\frak g_{p}$ is identified with the tangent space of $X$ at a point $x$ where $x$ is in the open $G$-orbit.

More precisely, let $\frak l_k$ and $U_k$ be eigenspaces that have eigenvalue $k$ under the action of $E_X:=E_{\alpha}$. Then,
\begin{eqnarray*}
\frak l=\bigoplus_{k=-\mu(\frak l)}^{\mu(\frak l)}\frak l_k \text{ and }  U=\bigoplus_{k=-\mu(U)}^{\mu(U)}U_k
\end{eqnarray*}
where $\mu(\frak l)$ and $\mu(U)$ are the largest numbers among the nonzero eigenvalues of the action of $E_X$ on $\frak l$ and $U$, respectively. Now give the gradation on $\frak g$ by shifting the above decompositions as follows:
 \begin{eqnarray*}
 \frak g_{-p}&=&\frak l_{-p}  \text{ for } -p<-1\\
 \frak g_{-1}&=&\frak l_{-1}+ U_{-\mu(U)} \\
 \frak g_0&=&(\frak l_0+ \mathbb C)\rhd U_{-\mu(U)+1} \\
 \frak g_p&=&\frak l_p+ U_{-\mu(U)+p+1} \text{ for } p \geq 1
 \end{eqnarray*}
such that $[\frak g_p, \frak g_q]\subset \frak g_{p+q} \text{ for } p, q\in\mathbb Z$.
\end{proposition}

To prove Proposition \ref{tangentsp of X}, we need to calculate the eigenvalues under $E_{\alpha}$-actions.

\begin{lemma}\label{gradation of X}Let $X=(L, \alpha, \beta)$ and $E_{\alpha}$ be the characteristic element associated with root $\alpha$. The Lie algebra of automorphism $\aut(X)=(\frak l+\mathbb C) \rhd U$ has eigenspace decomposition under the action of $E_X=E_{\alpha}$ {\rm(}$\mathbb C$ has zero eigenvalue{\rm)}. Let $\frak l_k$ and $U_k$ be eigenspaces that have eigenvalue $k$.
\begin{enumerate}
\item[\rm(1)] $(B_m, \alpha_{m-1}, \alpha_{m})$, $m>2$ where $U=V(\pi_m)$. Let $E_X=E_{\alpha_{m-1}}$ then
\begin{eqnarray*}
\frak l_{-2}+\frak l_{-1}&+&\frak l_{0}+\frak l_{1}+\frak l_{2}, \\
U_{-\frac{m-1}{2}}&+&U_{-\frac{m-1}{2}+1}+\cdots+U_{\frac{m-1}{2}-1}+U_{\frac{m-1}{2}},
\end{eqnarray*}
and $\dim U_{-\frac{m-1}{2}}=2$.
\item[\rm(2)] $(B_3, \alpha_{1}, \alpha_{3})$ where $U=V(\pi_3)$; let $E_X=E_{\alpha_1}$ then
\begin{eqnarray*}
\frak l_{-1}&+&\frak l_{0}+\frak l_{1}, \\
U_{-\frac{1}{2}}&+&U_{\frac{1}{2}},
\end{eqnarray*}
and $\dim U_{-\frac{1}{2}}=4$.
\item[\rm(3)] $(C_m, \alpha_{m}, \alpha_{m-1})$ where $U=V(\pi_1)$. Let $E_X=E_{\alpha_m}$ then
\begin{eqnarray*}
\frak l_{-1}&+&\frak l_{0}+\frak l_{1}, \\
U_{-\frac{1}{2}}&+&U_{\frac{1}{2}},
\end{eqnarray*}
and $\dim U_{-\frac{1}{2}}=m$.
\item[\rm(4)] $(C_m, \alpha_{i+1}, \alpha_{i})$, $m>2$, $i=1,\ldots, m-2$ where $U=V(\pi_1)$. Let $E_X=E_{\alpha_{i+1}}$ then
\begin{eqnarray*}
\frak l_{-2}+\frak l_{-1}&+&\frak l_{0}+\frak l_{1}+\frak l_{2}, \\
U_{-1}&+&U_{0}+U_{1},
\end{eqnarray*}
and $\dim U_{-1}=i+1$.
\item[\rm(5)] $(F_{4}, \alpha_{2}, \alpha_{3})$ where $\alpha_{2}$ is a long root and $U=V(\pi_4)$. Let $E_X=E_{\alpha_2}$ then
\begin{eqnarray*}
\frak l_{-3}+\frak l_{-2}+\frak l_{-1}&+&\frak l_{0}+\frak l_{1}+\frak l_{2}+\frak l_{3}, \\
U_{-2}&+&U_{-1}+U_{0}+U_{1}+U_{2},
\end{eqnarray*}
and $\dim U_{-2}=3$.
\item[\rm(6)] $(G_{2}, \alpha_{2}, \alpha_{1})$ where $U=V(\pi_1)$. Let $E_X=E_{\alpha_2}$ then
\begin{eqnarray*}
\frak l_{-2}+\frak l_{-1}&+&\frak l_{0}+\frak l_{1}+\frak l_{2}, \\
U_{-1}&+&U_{0}+U_{1},
\end{eqnarray*}
and $\dim U_{-1}=2$.
\end{enumerate}
Furthermore, $\frak l_k$ and $U_k$ are irreducible $L_0$-representations.
\end{lemma}
\begin{proof}It is calculated with basis elements from \cite{Wa} or \cite{OV}.\end{proof}

 Let $\frak l_{-}=\bigoplus_{p<0} \frak l_p$ and $U_{-}=U_{-\mu(U)}$. Then, $\frak m=\frak l_{-}+U_{-}$.

\begin{proof}[Proof of Proposition \ref{tangentsp of X}]
 Let $\tilde{X}$ be the blowup of $X$ along $Z$. Since the open $G$-orbit of $X$ is isomorphic to the open $G$-orbit of $\tilde X$, it is sufficient to show that $T_{x}\tilde X$ is identified with $\frak m=\frak l_{-}+U_{-}$ for any $x$ that is in the open $G$-orbit of $\tilde X$.

 According to Theorem \ref{Horospheical} and its proof, $G=(L\times \mathbb C^*)/C\ltimes U$, where $U$ is a $L$-representation space and $C$ is the centralizer. $\tilde{X}$ is a projective bundle over the $L$-orbit $Y\cong L/P_{\alpha}$ such that $U$ acts by translation on the fibers, where $P_{\alpha}$ is the parabolic subgroup of $L$ associated with the root $\alpha$. For any point $x \in Y$, the tangent directions of the $L$-action at $x$ are naturally identified with $\frak l_{-}\cong T_x Y$, and the other tangent directions are contained in $U$.

 Hence, we assume that $x$ is in $Y$ which contained in the open $G$-orbit of $\tilde X$ and choose the characteristic element $E_{\alpha}$ of $\frak l$ associated with the root $\alpha$ as $E_X$: the {\it grading element of $\frak g$}. Let $\frak l_k$ and $U_k$ be eigenspaces that have eigenvalue $k$ under the action of $E_X$. Then, according to Lemma \ref{gradation of X}, we see
\begin{eqnarray*}
\frak l=\bigoplus_{k=-\mu(\frak l)}^{\mu(\frak l)}\frak l_k \text{ and }  U=\bigoplus_{k=-\mu(U)}^{\mu(U)}U_k
\end{eqnarray*}
where $\mu(\frak l)$ and $\mu(U)$ are the largest numbers among the nonzero eigenvalues of the action of $E_X$ on $\frak l$ and on $U$, respectively.

Since $U_k$ is an irreducible $L_0$-module and we see that $[\frak l_{-1},U_k]=U_{k-1}$, if the tangent space of $X$ at $x$ contains $U_k$, it must contain $U_{k-1}$. We can easily check that the dimension $\dim X=\dim L/(P_{\alpha}\cap P_{\beta})+1$ equals $\dim L/P_{\alpha}+\dim U_{-}$ in all cases. Hence, if one give the gradation on $\frak g$ by shifting based on $\frak g_{-1}=\frak l_{-1}+ U_{-\mu(U)}$, then the tangent space $T_{x}X$ at $x$ is identified with $\frak m=\frak l_{-}+U_{-}$. Since the gradation is given by $E_X$, it is clear that $[\frak g_p, \frak g_q]\subset \frak g_{p+q} \text{ for } p, q\in\mathbb Z$.

\end{proof}

Let $\frak l_{> 0}=\frak l_{1}+ \cdots+\frak l_{\mu(\frak l)}$ and $\frak l_{\geq 0}=\frak l_{0}+\frak l_{> 0}$.
Let $U_{\tilde i}=:U_{-\mu(U)+1+i}$ for $i=-1,\cdots,l$ where $l= 2\mu(U)-1$. For example, $U_{\widetilde{-1}}=U_{-\mu(U)}$ which is $U_-$, $U_{\tilde 0}=U_{-\mu(U)+1}$, and $U_{\tilde l}=U_{\mu(U)}$. Let $U_{+}=\bigoplus_{1\leq i\leq 2\mu(U)-1}U_{\tilde i}$, and $U_{\geq \tilde 0}=U_{\tilde 0}+U_{+}=\bigoplus_{0\leq i\leq 2\mu(U)-1}U_{\tilde i}$. Let $\frak g_{\geq 0}=\frak l_{\geq 0}+U_{\geq \tilde 0}+\mathbb C$.

\begin{lemma}\label{fun} Let $\frak g=\bigoplus_p \frak g_{p}$ be a graded Lie algebra given in Proposition \ref{tangentsp of X}. Then,
\begin{enumerate}
    \item[\rm(1)] $\frak m=\bigoplus_{p<0}\frak g_{p}$ is fundamental, i.e., $\frak g_p=[\frak g_{p+1}, \frak g_{-1}]$ for $p<-1$;
    \item[\rm(2)] If $z \in \frak l_{\geq 0}+U_{\geq \tilde 0}$ satisfies $[z, \frak l_{-1}]=0$, then $z=0$. Further, if $z\in \frak g_{\geq 0}$ satisfies $[z, \frak g_{-1}]=0$, then $z=0$;
    \item[\rm(3)] for any nonzero vector $u \in U_{\tilde 0}$, the dimension of the subspace $[\frak l_{-1},u] \subset U_{-}$ is greater than or equal to $2$.
\end{enumerate}
 \end{lemma}

\begin{proof}
\begin{enumerate}
\item  The gradation of $\frak l=\bigoplus_{p\in \mathbb Z} \frak l_p$ associated with $\alpha$ satisfies $\frak l_p=[\frak l_{p+1}, \frak l_{-1}]$ for $p<-1$. For $p=-2$, we see
\begin{eqnarray*}
\frak g_{-2}=\frak l_{-2} &=&[\frak l_{-1},\frak l_{-1}]\\
            &=&[\frak l_{-1}+U_-,\frak l_{-1}+U_-] \text{ because } [\frak l_{-1}+U_{-},U_{-}]=0\\
            &=&[\frak g_{-1}, \frak g_{-1}].
\end{eqnarray*}
For $p<-2$, we see
\begin{eqnarray*}
\frak g_{p}=\frak l_{p} &=&[\frak l_{p+1},\frak l_{-1}]\\
            &=&[\frak l_{p+1},\frak l_{-1}+U_-] \text{ because } [\frak l_{p+1},U_{-}]=0\\
            &=&[\frak g_{p+1}, \frak g_{-1}].
\end{eqnarray*}
Hence, the graded Lie algebra $\frak m=\bigoplus_{p<0}\frak g_{p}$ is fundamental.

\item  Assume $z\in U_{\geq \tilde 0}$ satisfies $[z, \frak l_{-1}]=0$. Since $[U_{\widetilde{k-1}}, \frak l_{1}]=U_{\tilde{k}}$ for $k \geq 0$,
  \begin{eqnarray*}
  0=\{[z, \frak l_{-1}], U_{\widetilde{k-1}} \} =\{z, [\frak l_{1}, U_{\widetilde{k-1}}] \}=\{z, U_{\widetilde{k}} \},
  \end{eqnarray*}
  which implies that $z=0$. By Lemma 1.3 of \cite{Ta79}, if $z\in \frak l _{\geq 0}$ satisfies $[z, \frak l_{-1}]=0$, then $z=0$. Hence, if $z\in \frak l _{\geq 0}+U_{\geq \tilde 0}$ satisfies $[z, \frak l_{-1}]=0$, then $z=0$.

 Since $\frak l_{-1} \subset \frak g_{-1}$, if $z\in \frak l _{\geq 0}+U_{\geq \tilde 0}$ satisfies $[z, \frak g_{-1}]=0$, then $z=0$.

 If $z \in \mathbb C$ satisfies $[z,\frak g_{-1}]=0$, then $[z,U_{-}]=z.U_{-}=0$ because $[\mathbb C,\frak l_{-1}]=0$. For $z \in \mathbb C$, if $z.U_{-}=0$, then $z=0$. Hence, if $z\in \frak g_{\geq 0}=\frak l _{\geq 0}+U_{\geq \tilde 0}+\mathbb C$ satisfies $[z, \frak g_{-1}]=0$, then $z=0$.

\item The action
\begin{eqnarray*}
\frak l_{-1} \times U_{\tilde 0} &\rightarrow & U_{-} \\
 (l,u) &\mapsto & [l,u]
\end{eqnarray*}
is described as following list up to scalar. The following list is from weights and weight diagrams(\cite{FH}) of the irreducible $\frak l_0$-representations on $\frak l_k$ and $U_k$. Let $R \omega(T)$ be the irreducible representation of type $T$ with the highest weight $\omega$. Let $*$ be the usual complex conjugation.
\begin{enumerate}

\item $(B_m, \alpha_{m-1}, \alpha_{m})$, $m>2$, where $U=V(\pi_m)$. Let $R\pi_1(A_{1})=W$ be the standard representation of $A_1$. Let $R\pi_1(A_{m-2})=Q$ be the standard representation of $A_{m-2}$. Then, $\dim W=2$, $W^*=W$, $\dim Q = m-1$ and
\begin{eqnarray*}
&\frak l_{-1}=R\pi_1(A_{m-2})^*\otimes R2\pi_1(A_{1})^*=Q^* \otimes \Sym^2W^* \\
&U_{-}=R\pi_1(A_1)^*=W^*\\
&U_{\tilde 0}=R\pi_1(A_{m-2})\otimes R\pi_1(A_1)=Q \otimes W.
\end{eqnarray*}

The action $\frak l_{-1} \times U_{\tilde 0} \rightarrow  U_{-}$ is given as follows, for $w_1, w_2 \in W$ such that $W=\langle w_1, w_2 \rangle$ and $q \in Q$:
\begin{eqnarray*}
(Q^* \otimes \Sym^2W^*) \times (Q \otimes W) &\rightarrow & W^* \\
    (q^*\otimes w^*_1\odot w^*_2 , q \otimes w_1) &\mapsto & q^*(q)w^*_1\odot w^*_2(w_1)=w^*_2\\
    (q^*\otimes w^*_1\odot w^*_1 , q \otimes w_1) &\mapsto & q^*(q)w^*_1\odot w^*_1(w_1)=2w^*_1.
\end{eqnarray*}

\item$(B_3, \alpha_{1}, \alpha_{3})$ where $U=V(\pi_3)$. Let $W$ be the spin representation of $B_2$. Let $V$ be the standard representation of $B_2$. Then, $V^*=V$, $\dim V=5$, $\dim W=4$, $W=W^*$ and
\begin{eqnarray*}
&\frak l_{-1}=R\pi_1(B_2)^*=V^*\\
&U_{-}=R\pi_2(B_2)^*=W^*\\
&U_{\tilde 0}=R\pi_2(B_2)=W.
\end{eqnarray*}

The action $\frak l_{-1} \times U_{\tilde 0} \rightarrow  U_{-}$ is given by the following:
\begin{center}
\begin{tabular}{|>{$}c<{$}|>{$}c<{$}|>{$}c<{$}|>{$}c<{$}|>{$}c<{$}|>{$}c<{$}|}
 \hline
   \times  & v^*_1 & v^*_2 & v^*_3 & v^*_4 & v^*_5 \\
 \hline
  w_1 & w^*_4 & w_3^* & w_2^* & \cdot & \cdot \\
 \hline
  w_2 & w_3^* & \cdot & w_1^* & w_4^* & \cdot \\
 \hline
  w_3 & w_2^* & w_1^* & \cdot & \cdot & w_4^* \\
 \hline
  w_4 & w_1^* & \cdot & \cdot  & w_3^*  & w_2^* \\
  \hline
\end{tabular}
\end{center}
 where $\{w_1, w_2, w_3, w_4\}$ is a basis  of $W$ and $\{v_{1}, v_{2}, v_{3}, v_{4}, v_{5}\}$ is a basis  of $V$.

\item$(C_m, \alpha_{m}, \alpha_{m-1})$ where $U=V(\pi_1)$. Let $W$ be the standard representation of $A_{m-1}$. Then, $\dim W = m$ and
\begin{eqnarray*}
&\frak l_{-1}=R2\pi_1(A_{m-1})^*=\Sym^2W^* \\
&U_{-}=R\pi_1(A_{m-1})^*=W^*\\
&U_{\tilde 0}=R\pi_1(A_{m-1})=W.
\end{eqnarray*}

The action $\frak l_{-1} \times U_{\tilde 0} \rightarrow  U_{-}$ is given as follows, for the orthonormal basis $w_i, w_j, w_k \in W$:
\begin{eqnarray*}
\Sym^2W^* \times W &\rightarrow & W^* \\
    (w_i^* \odot w_j^* , w_k) &\mapsto & (w_i^* \odot w_j^*) (w_k)= \delta_{jk}w_i^*+\delta_{ik}w_j^*.
\end{eqnarray*}

\item$(C_m, \alpha_{i+1}, \alpha_{i})$, $m>2$, $i=1,\ldots, m-2$, where $U=V(\pi_1)$. Let $W$ be the standard representation of $A_{i}$ and let $Q$ be the standard representation of $C_{m-i-1}$. Then, $\dim W =i+1$, $\dim Q = 2m-2i-2$ and
\begin{eqnarray*}
&\frak l_{-1}=R\pi_1(A_{i})^*\otimes R\pi_1(C_{m-i-1})^*=W^*\otimes Q^* \\
&U_{-}=R\pi_1(A_{i})^*=W^*\\
&U_{\tilde 0}=R\pi_1(C_{m-i-1})=Q.
\end{eqnarray*}

The action $\frak l_{-1} \times U_{\tilde 0} \rightarrow  U_{-}$ is given as follows, for $q \in Q$ and $w \in W$:
\begin{eqnarray*}
(W^*\otimes Q^*) \times Q &\rightarrow & W^* \\
    (w^* \otimes q^*, q) &\mapsto & w^*  q^*(q).
\end{eqnarray*}

\item$(F_{4}, \alpha_{2}, \alpha_{3})$ where $\alpha_{2}$ is a long root and $U=V(\pi_4)$. Let $W$ be the standard representation of $A_{1}$ and let $V$ be the standard representation of $A_{2}$. Then, $\dim V = 3$, $\dim W =2$, $W^*=W$ and
\begin{eqnarray*}
&\frak l_{-1}=R2\pi_1(A_{2})\otimes R\pi_1(A_{1})=\Sym^2V \otimes W \\
&U_{-}=R\pi_1(A_{2})=V\\
&U_{\tilde 0}=R\pi_1(A_{2}) ^*\otimes R\pi_1(A_{1})^*=V^* \otimes W^*.
\end{eqnarray*}

By the action $\frak l_{-1} \times U_{\tilde 0} \rightarrow  U_{-}$, for $v_1, v_2, v_3 \in V$ and $w \in W$,
\begin{eqnarray*}
(\Sym^2V\otimes W) \times (V^* \otimes W^*) &\rightarrow & V \\
    (v_1\odot v_1 \otimes w, v_1^*\otimes w^*) &\mapsto & v_1\odot v_1 (v_1^*) w(w^*)=2v_1\\
    (v_1\odot v_2 \otimes w, v_1^*\otimes w^*) &\mapsto & v_1\odot v_2 (v_1^*) w(w^*)=v_2\\
    (v_1\odot v_3 \otimes w, v_1^*\otimes w^*) &\mapsto & v_1\odot v_3 (v_1^*) w(w^*)=v_3.
\end{eqnarray*}

\item$(G_{2}, \alpha_{2}, \alpha_{1})$ where $U=V(\pi_1)$. Let $W$ be the standard representation of $A_{1}$. Then, $\dim W =2$, $W^*=W$ and
\begin{eqnarray*}
&\frak l_{-1}=R3\pi_1(A_{1})^*=\Sym^3W^* \\
&U_{-}=R\pi_1(A_{1})^*=W^*\\
&U_{\tilde 0}=R2\pi_1(A_{1})=\Sym^2W.
\end{eqnarray*}

By the action $\frak l_{-1} \times U_{\tilde 0} \rightarrow  U_{-}$,  for $w_1, w_2 \in W$,
\begin{eqnarray*}
\Sym^3W^* \times \Sym^2W &\rightarrow & W^* \\
    (w_1^* \odot w_2^* \odot w_1^*, w_1 \odot w_2) &\mapsto & w_1^* \odot w_2^* \odot w_1^*(w_1 \odot w_2)= 2w_1^*\\
    (w_1^* \odot w_2^* \odot w_2^*, w_1 \odot w_2) &\mapsto & w_1^* \odot w_2^* \odot w_2^*(w_1 \odot w_2)= 2w_2^*\\
    (w_1^* \odot w_1^* \odot w_1^*, w_1 \odot w_1) &\mapsto & w_1^* \odot w_1^* \odot w_1^*(w_1 \odot w_1)= 3w_1^*\\
    (w_1^* \odot w_1^* \odot w_2^*, w_1 \odot w_1) &\mapsto & w_1^* \odot w_1^* \odot w_2^*(w_1 \odot w_1)= 2w_2^*.
\end{eqnarray*}
\end{enumerate}

From the above list of actions $\frak l_{-1} \times U_{\tilde 0} \rightarrow  U_{-}$, we easily see that for a nonzero vector $u \in U_{\tilde 0}$, the dimension of the subspace $[\frak l_{-1},u] \subset U_{-}$ is greater than or equal to $2$.
\end{enumerate}

\end{proof}

\section{Vanishing cohomologies}\label{section3}
Let $X$ be a smooth nonhomogeneous projective horospherical variety $(L, \alpha, \beta)$ of Picard number one. Let $G=\Aut(X)$ and let $\frak g=(\frak l+\mathbb C) \rhd U$ be the corresponding Lie algebra. By Proposition \ref{tangentsp of X}, we can give a gradation of $\frak g=\bigoplus_p \frak g_p$ such that the graded Lie algebra $\frak m=\bigoplus_{p<0}\frak g_{p}$ is identified with the tangent space of $X$ at a point $x$, where $x$ is in the open $G$-orbit. Let $\frak l_{-}=\bigoplus_{p<0} \frak l_p$, $\frak l_{> 0}=\frak l_{1}+ \cdots+\frak l_{\mu(\frak l)}$ and $\frak l_{\geq 0}=\frak l_{0}+\frak l_{> 0}$. Let $U_{\tilde i}=:U_{-\mu(U)+1+i}$ for $i=-1,\cdots,l$ where $l= 2\mu(U)-1$. Let $U_{-}=U_{\widetilde{-1}}$, $U_{+}=\bigoplus_{1\leq i\leq 2\mu(U)-1}U_{\tilde i}=U_{-\mu(U)+2}+ \cdots+U_{\mu(U)}$, and  $U_{\geq \tilde 0}=U_{\tilde 0}+U_{+}$. Let $\frak m = \frak l_- + U_-$, $\frak g_{> 0}=\frak l_{> 0}+U_{+}$, and $\frak g_{\geq 0}=\frak l_{\geq 0}+U_{\geq \tilde 0}+\mathbb C$, and let $\frak m'$ be the dual of $\frak m$.

\begin{proposition}\label{H^1(g)}
Let $\frak g= (\frak l+ \mathbb C) \rhd U$ and $\frak m = \frak l_- + U_-$.
Assume that
\begin{enumerate}
  \item[\rm(1)]  if $z \in \frak l_{\geq 0}+U_{\geq \tilde 0}$ satisfies $[\frak l_{-}, z]=0$, then $z=0$;
  \item[\rm(2)]  for any vector $u \in U_{\tilde 0}$, if the dimension of the subspace $[\frak l_{-1},u] \subset U_{-}$ is less than or equal to $1$, then $u=0$.
\end{enumerate}
For $p>0$, if $H^{p,1}(\frak l_-, \frak l)=0$ and $H^{p,1}(\frak l_-, U)=0$, then $H^{p,1}(\frak m, \frak g)=0$.
\end{proposition}

\begin{proof}
For $p>0$, $C^{p,1}(\frak m, \frak g) \subset \frak m' \otimes \frak g$. Let $\phi \in C^{p,1}(\frak m, \frak g)$ such that $\partial \phi=0$. We will show that there exist $\psi \in \frak g$ such that
\begin{eqnarray*} \partial \psi =\phi. \end{eqnarray*}
Write $\phi=\phi_{\frak l}+\phi_{\mathbb C} + \phi_U$, where $\phi_{\frak l} \in \frak m' \otimes \frak l$, $\phi_{\mathbb C} \in \frak m' \otimes \mathbb C$, and $\phi_{U} \in \frak m' \otimes U$.

For any $X^{\frak l_-}, Y^{\frak l_-} \in \frak l_-$ and $X^{U_-}, Y^{U_-} \in U_-$, we have
\begin{eqnarray*}
0&=& \partial \phi(X^{\frak l_-} + X^{U_-}, Y^{\frak l_-}+Y^{U_-}) \\
&=& [X^{\frak l_-} + X^{U_-}, \phi(Y^{\frak l_-} +Y^{U_-})] - [Y^{\frak l_-} + Y^{U_-},\phi(X^{\frak l_-} +X^{U_-})] - \phi([X^{\frak l_-}, Y^{\frak l_-}]) \\
 &&\text{ because }[\frak l_- + U_{-}, U_{-}]=0 \\
&=& \left\{ [X^{\frak l_-}, \phi_{\frak l + \mathbb C}(Y^{\frak l_-} + Y^{U_-})] - [Y^{\frak l_-},\phi_{\frak l+\mathbb C}(X^{\frak l_-} + X^{U_-})] - \phi_{\frak l + \mathbb C}([X^{\frak l_-}, Y^{\frak l_-}])\right\} \\
&&+ \left\{[X^{U_-}, \phi_{\frak l + \mathbb C}(Y^{\frak l_-} + Y^{U_-})] - [Y^{U_-}, \phi_{\frak l+ \mathbb C}(X^{\frak l_-} + X^{U_-})]\right. \\
 &&+ \left.[X^{\frak l_-}, \phi_U(Y^{\frak l_-} +Y^{U_-})] - [Y^{\frak l_-},\phi_U(X^{\frak l_-} + X^{U_-})] - \phi_U([X^{\frak l_-}, Y^{\frak l_-}])  \right\}.
\end{eqnarray*}
Thus,
\begin{eqnarray*}
 0&\stackrel{(\star)}{=}& \phi_{\mathbb C}([X^{\frak l_-}, Y^{\frak l_-}]) \\
 0&\stackrel{(\ast)}{=}& [X^{\frak l_-}, \phi_{\frak l }(Y^{\frak l_-} + Y^{U_-})] - [Y^{\frak l_-},\phi_{\frak l}(X^{\frak l_-} + X^{U_-})] - \phi_{\frak l}([X^{\frak l_-}, Y^{\frak l_-}])\\
 0&\stackrel{(\diamond)}{=}& [X^{U_-}, \phi_{\frak l}(Y^{\frak l_-} + Y^{U_-})] - [Y^{U_-}, \phi_{\frak l}(X^{\frak l_-} + X^{U_-})] \\
 &&  + [X^{U_-}, \phi_{\mathbb C}(Y^{\frak l_-} + Y^{U_-})] - [Y^{U_-}, \phi_{\mathbb C}(X^{\frak l_-} + X^{U_-})] \\
 &&  + [X^{\frak l_-}, \phi_U(Y^{\frak l_-} +Y^{U_-})] - [Y^{\frak l_-},\phi_U(X^{\frak l_-} + X^{U_-})] - \phi_U([X^{\frak l_-}, Y^{\frak l_-}]).
\end{eqnarray*}

Put $X^{U_-}=Y^{U_-}=0$ into ($\ast$) to get
\begin{eqnarray} \label{1}
[X^{\frak l_-}, \phi_{\frak l}(Y^{\frak l_-})] - [Y^{\frak l_-},\phi_{\frak l}(X^{\frak l_-})] - \phi_{\frak l}([X^{\frak l_-}, Y^{\frak l_-}])=0.
\end{eqnarray}
Put $Y^{\frak l_-}=0$ into ($\ast$) to get
\begin{eqnarray} \label{2}
[X^{\frak l_-}, \phi_{\frak l}(Y^{U_-})]=0.
\end{eqnarray}
Put $Y^{\frak l_-}=0$ and $X^{U_-}=0$ into ($\diamond$) to get
\begin{eqnarray} \label{3}
[X^{\frak l_-}, \phi_U(Y^{U_-})]-[Y^{U_-},\phi_{\frak l}(X^{\frak l_-})]-[Y^{U_-},\phi_{\mathbb C}(X^{\frak l_-})]=0.
\end{eqnarray}
Put $Y^{\frak l_-}=0$ and $X^{\frak l_-}=0$ into ($\diamond$) to get
\begin{eqnarray} \label{4}
&&[X^{U_-},\phi_{\frak l}(Y^{U_-})]-[Y^{U_-},\phi_{\frak l}(X^{U_-})] \\
&&+[X^{U_-},\phi_{\mathbb C}(Y^{U_-})]-[Y^{U_-},\phi_{\mathbb C}(X^{U_-})]=0.\nonumber
\end{eqnarray}
Put $X^{U_-}=Y^{U_-}=0$ into ($\diamond$) to get
\begin{eqnarray} \label{5}
[X^{\frak l_-},\phi_U(Y^{\frak l_-})] - [Y^{\frak l_-},\phi_U(X^{\frak l_-})] - \phi_U([X^{\frak l_-}, Y^{\frak l_-}])=0.
\end{eqnarray}

By (\ref{1}) and (\ref{5}), we have $\partial(\phi_{\frak l}+\phi_{U})(X^{\frak l_-},Y^{\frak l_-})=0$. By hypothesis, $H^{p,1}(\frak l_-, \frak l)=0$ and $H^{p,1}(\frak l_-, U)=0$ for $p>0$.  Then, there exist $\psi=\psi_{\frak l}+\psi_{U}$, where $\psi_{\frak l} \in \frak l$ and $\psi_{U} \in U$ such that
\begin{eqnarray} \label{6}
\partial\psi(X^{\frak l_-}) = (\phi_{\frak l}+ \phi_{U})(X^{\frak l_-}).
\end{eqnarray}

In (\ref{2}), since $X^{\frak l_-} \in \frak l_{-}$ is arbitrary, by assumption (1), we have
\begin{eqnarray}\label{2-2}
\phi_{\frak l}(Y^{U_-})=0.
\end{eqnarray}
In (\ref{4}), by (\ref{2-2}), we see that
\begin{eqnarray} \label{4-2}
[X^{U_-}, \phi_{\mathbb C}(Y^{U_-})]-[Y^{U_-},\phi_{\mathbb C}(X^{U_-})]=0.
\end{eqnarray}
Equation (\ref{4-2}) is also valid for the two linearly independent vectors $X^{U_-}$ and $Y^{U_-}$, and $\mathbb C$ act on $U$ as scalars. Hence,
\begin{eqnarray} \label{4-3}
\phi_{\mathbb C}(X^{U_-})=0.
\end{eqnarray}

We will show that $\partial \psi = \phi$ where $\phi=\phi_{\frak l}+ \phi_{U}+\phi_{\mathbb C}$. By ($\star$), (\ref{6}), (\ref{2-2}), and (\ref{4-3}), it suffices to show that
$\phi_{\mathbb C}(X^{\frak l_-})=0$ for all $X^{\frak l_{-}} \in \frak l_{-1}$ and $\partial \psi(X^{U_-}) = \phi_{U}(X^{U_-})$ for all $X^{U_-} \in U_-$.

Let $X^{\frak l_{-}} \in \frak l_{-1}$. In (\ref{3}), we have
\begin{eqnarray*}
[Y^{U_-},\phi_{\mathbb C}(X^{\frak l_-})]&=&[X^{\frak l_-}, \phi_U(Y^{U_-})]-[Y^{U_-},\phi_{\frak l}(X^{\frak l_-})]\\
                                    &=&[X^{\frak l_-}, \phi_U(Y^{U_-})]-[Y^{U_-},[X^{\frak l_-},\psi_{\frak l}]] \text{ because } \phi_{\frak l}(X^{\frak l_-})= \partial \psi_{\frak l}(X^{\frak l_-})\nonumber \\
                                    &=&[X^{\frak l_-}, \phi_U(Y^{U_-})]-[X^{\frak l_-},[Y^{U_-},\psi_{\frak l}]] \text{ because }[\frak l_-, U_{-}]=0 \nonumber\\
                                    &=&[X^{\frak l_-},(\phi_U(Y^{U_-})-[Y^{U_-},\psi_{\frak l}])].
\end{eqnarray*}
Hence, for $X^{\frak l_{-}} \in \frak l_{-1}$ and $Y^{U_-} \in U_-$,
\begin{eqnarray}\label{3-2}
[Y^{U_-},\phi_{\mathbb C}(X^{\frak l_-})]=[X^{\frak l_-},\phi_U(Y^{U_-})-\partial \psi_{\frak l}(Y^{U_-})].
\end{eqnarray}

Because $\mathbb C$ act on $U$ as a scalar, left side of (\ref{3-2}) is in $U_-$ and hence, the $\phi_U(Y^{U_-})-\partial \psi_{\frak l}(Y^{U_-})$ in the bracket of right side is in $U_{\tilde 0}$

From the decomposition $\phi_U(Y^{U_-})-\partial \psi_{\frak l}(Y^{U_-})=\bigoplus_{i}(\phi_U(Y^{U_-})-\partial \psi_{\frak l}(Y^{U_-}))_{\tilde i}$, where $(\phi_U(Y^{U_-})-\partial \psi_{\frak l}(Y^{U_-}))_{\tilde i} \in U_{\tilde i}$, we have
\begin{eqnarray}\label{3-2-2}
[Y^{U_-},\phi_{\mathbb C}(X^{\frak l_-})]&=&[X^{\frak l_-},(\phi_U(Y^{U_-})-\partial \psi_{\frak l}(Y^{U_-}))_{\tilde 0}] \\
 0&=&[X^{\frak l_-},(\phi_U(Y^{U_-})-\partial \psi_{\frak l}(Y^{U_-}))_{\tilde i}] \text{ for } \tilde i \neq \tilde 0.\nonumber
\end{eqnarray}
For $X^{\frak l_-} \in \{X^{\frak l_-} \in \frak l_{-1} | \phi_{\mathbb C}(X^{\frak l_-})=0 \}$,
\begin{eqnarray*} \label{3-3}
0&=&[X^{\frak l_-},(\phi_U(Y^{U_-})-\partial \psi_{\frak l}(Y^{U_-}))_{\tilde 0}].
\end{eqnarray*}
Since $\{X^{\frak l_-} \in \frak l_{-1} | \phi_{\mathbb C}(X^{\frak l_-})=0 \} \subset \frak l_{-1}$ is a hyperplane or $\frak l_-$, by assumption (2),
\begin{eqnarray} \label{3-4}
0&=&(\phi_U(Y^{U_-})-\partial \psi_{\frak l}(Y^{U_-}))_{\tilde 0}.
\end{eqnarray}
By (\ref{3-2-2}) and (\ref{3-4}), for any $Y^{U_-} \in U_-$ and $X^{\frak l_{-}} \in \frak l_{-1}$,
\begin{eqnarray*} \label{3-4-1}[Y^{U_-},\phi_{\mathbb C}(X^{\frak l_-})]=0.\end{eqnarray*}
Hence, for any $X^{\frak l_{-}} \in \frak l_{-1}$,
\begin{eqnarray} \label{3-5}\phi_{\mathbb C}(X^{\frak l_-})=0.\end{eqnarray}

By (\ref{3-5}), equation (\ref{3}) becomes
\begin{eqnarray} \label{3-6}
[X^{\frak l_-}, \phi_U(Y^{U_-})]=[Y^{U_-},\phi_{\frak l}(X^{\frak l_-})]=[Y^{U_-},\partial \psi_{\frak l}(X^{\frak l_-})].
\end{eqnarray}
Write $\partial \psi=(\partial \psi)_{\frak l}+(\partial \psi)_{\mathbb C} + (\partial \psi)_U$ where $(\partial \psi)_{\frak l} \in \frak m' \otimes \frak l$, $(\partial \psi)_{\mathbb C} \in \frak m' \otimes \mathbb C$ and $(\partial \psi)_U \in \frak m' \otimes U$.
Then,
\begin{eqnarray*}
\partial \psi (X^{\frak l_-}) &= [X^{\frak l_-},\psi] &=[X^{\frak l_-}, \psi_{\frak l}] + [X^{\frak l_-}, \psi_{U}]\\
\partial \psi (X^{U_-}) &= [X^{U_-}, \psi] &=  [X^{U_-},\psi_{\frak l}].
\end{eqnarray*}
Thus,
\begin{eqnarray}
\label{3-6-1}(\partial \psi)_{\mathbb C} &=&0\\
\label{3-6-2}(\partial \psi)_{\frak l}(X^{U_-}) &=& 0,
\end{eqnarray}
and
\begin{eqnarray*}
(\partial \psi)_{\frak l}(X^{\frak l_-}) &=& [X^{\frak l_-}, \psi_{\frak l}] \\
(\partial \psi)_U(X^{\frak l_-}) &=& [X^{\frak l_-}, \psi_U] \\
(\partial \psi)_U(X^{U_-}) &=&[X^{U_-}, \psi_{\frak l}].
\end{eqnarray*}
In particular, this implies that
\begin{eqnarray*}
[X^{\frak l_-},(\partial \psi)_U(X^{U_-})] &=& [X^{\frak l_-},[X^{U_-}, \psi_{\frak l}]] \\
&=& [X^{U_-},[X^{\frak l_-}, \psi_{\frak l}]] \text{ because } [\frak l_-, U_-]=0 \\
&=& [X^{U_-}, (\partial \psi)_{\frak l}(X^{\frak l_-})].
\end{eqnarray*}
Hence, by (\ref{3-6}),
\begin{eqnarray}\label{3-7}
[X^{\frak l_-},(\partial \psi)_U(X^{U_-})] = [X^{\frak l_-}, \phi_U(X^{U_-})].
\end{eqnarray}
Since $X^{\frak l_-} \in \frak l_{-}$ is arbitrary in (\ref{3-7}), by assumption (1), we have
\begin{eqnarray*}
(\partial \psi)_{U}(X^{U_-})= \phi_{U}(X^{U_-}).
\end{eqnarray*}
Hence, by (\ref{3-6-1}) and (\ref{3-6-2}),
\begin{eqnarray}\label{3-8}
(\partial \psi)(X^{U_-})= \phi_{U}(X^{U_-}).
\end{eqnarray}

It follows that $\partial \psi =\phi$. Therefore, $H^{p,1}(\frak m, \frak g) =0$ for any positive integer $p$.
\end{proof}

\begin{lemma}\label{H^1(U)}$H^{p,1}(\frak l_-, U)=0$ for $p>0$.\end{lemma}

\begin{proof}
Let $z_{\alpha} \in \frak l$ be a root vector associated with the simple root $\alpha$. Let $\sigma_{\alpha}$ be the simple reflection associated with $\alpha$. Let $\lambda$ be the highest weight of $\frak l$ on the irreducible representation $U$. Then, $-\lambda$ is the lowest weight on $U$. Let $u_{-\sigma_{\alpha}(\lambda)} \in U$ be the weight vector with weight $-\sigma_{\alpha}(\lambda)$.

By Theorem 5.15 of \cite{Ko}, we have
\begin{eqnarray*}H^1(\frak l_{-}, U) = \mathcal H^{\xi_\sigma},\end{eqnarray*}
where $\mathcal H^{\xi_\sigma}$ is the irreducible $\frak l_0$-module with the lowest weight vector $z_{\alpha}' \otimes u_{-\sigma_{\alpha}(\lambda)}$ having weight $\xi_\sigma =-(\sigma_{\alpha}(\lambda)+\alpha)$.

Since $z_{\alpha}' \otimes u_{-\sigma_{\alpha}(\lambda)} \in \frak l_{-1}' \otimes U_{-}$ and $\mathcal H^{\xi_\sigma}=\frak l_{-1}' \otimes U_{-}=H^{0,1}(\frak l_-, U)$, we see that $H^{p,1}(\frak l_-, U)=0$ for $p>0$.
\end{proof}

\begin{proposition}\label{H^1(')} Let $\frak g=\bigoplus_p \frak g_{p}$ be a graded Lie algebra given in Proposition \ref{tangentsp of X}. $H^{p,1}(\frak m, \frak g)=0$ for $p > 0$.\end{proposition}

\begin{proof}
By Lemma of \cite{Ya}, $H^{p,1}(\frak l_{-}, \frak l)=0$ for $p>0$. By Lemma \ref{H^1(U)}, $H^{p,1}(\frak l_-, U)=0$ for $p>0$. By (2) and (3) of Lemma \ref{fun}, the two assumptions of Proposition \ref{H^1(g)} are satisfied. Hence, $H^{p,1}(\frak m, \frak g)=0$ for $p>0$.
\end{proof}

\begin{proposition}\label{Prolongation} Let $\frak g=\bigoplus_p \frak g_{p}$ be a graded Lie algebra given in Proposition \ref{tangentsp of X}. The Lie algebra $\frak g$ is the prolongation of $(\frak m, \frak g_0)$. \end{proposition}
\begin{proof}
By Lemma \ref{fun}, $\frak m=\bigoplus_{p<0}\frak g_{p}$ is fundamental, i.e., $\frak g_p=[\frak g_{p+1}, \frak g_{-1}]$ for $p<-1$ and if $z\in \frak g_p$ for $p\geq 0$ satisfies $[z, \frak g_{-1}]=0$, then $z=0$. By Proposition \ref{H^1(')}, $H^{p,1}(\frak m, \frak g)=0$ for $p\geq 1$.
Hence, according to Lemma \ref{C.prolongation}, the graded Lie algebra $\frak g=\bigoplus\frak g_p$ of Proposition \ref{tangentsp of X} is the prolongation of $(\frak m, \frak g_0)$.
\end{proof}

\section{Geometric structures}\label{section4}

Let $X$ be a smooth nonhomogeneous projective horospherical variety of Picard number one. Let $X^o$ be the open orbit of $X$ with respect to $G=\Aut(X)$. Let $\frak g$ be the Lie algebra of $G$. We recall, from Proposition \ref{tangentsp of X}, that there is a gradation of $\frak g$ such that $\frak m=\bigoplus_{p<0} \frak g_p$ is fundamental and
\begin{eqnarray}\label{fund} \iota: T_xX^o\simeq\frak m
\end{eqnarray}
for a base point $x\in X^o$.

\begin{definition}\label{Flat Horo}Let $(X^o,E)$ be the regular differential system of type $\frak m$ derived from the subbundle $E$ of $TX^o$, where $E_x$ corresponds to $\frak g_{-1}$ under the identification $T_xX^o\simeq\frak m$ for a base point $x\in X^o$. Let $G_0\subset G_{0}(\frak m)$ be the Lie subgroup corresponding to $\frak g_0$. Let $\mathscr R$ be the frame bundle of $(X^o,E)$. Then, $\mathscr R$ is isomorphic to $G\times_H G_0(\frak m)$, where $H=H(\frak m, G_0)$.

The $G_0$-subbundle $\mathscr P$ of $\mathscr R$, which is isomorphic to the $G_0$-subbundle $G\times_H G_0$ of $G\times_H G_0(\frak m)$, is a $G_0$-structure on $(X^o,E)$ referred to as \emph{the standard geometric structure} on $X$.
\end{definition}

\begin{definition}\label{geometric ST} Let $\frak m=\bigoplus_{p<0} \frak g_p$ be the fundamental graded Lie algebra given by Proposition \ref{tangentsp of X}, which satisfies (\ref{fund}). Let $M$ be a projective manifold. A \emph{distribution} $D$ is a subbundle of $T(M)$, which is defined on outside of a subvariety $\Sing(D)$ of $M$. Suppose that
\begin{enumerate}
  \item any point $x \in M^o:= M-\Sing(D)$ is general and $\Sing(D)$ has codimension at least two;
  \item $(M^o,D)$ is a regular filtered manifold derived from a regular differential system of type $\frak m$.
\end{enumerate}
A $G_0$-structure on $(M^o,D)$ is called a \emph{geometric structure of $M$ modeled on $X$}(See Definition \ref{G_0 structure}).

Two geometric structures of $M_1$ and $M_2$ modeled on $X$ are \emph{locally equivalent} if the $G_0$-structure on $({M_1}^o,D_1)$ and the $G_0$-structure on $({M_2}^o,D_2)$ are locally equivalent. A geometric structure modeled on $X$ is \emph{locally flat} if it is locally equivalent to the standard geometric structure on $X$.
\end{definition}

\begin{proposition}\label{Cartan connection}Let $X$ be a smooth nonhomogeneous horospherical variety $(L, \alpha, \beta)$ of Picard number one. Let $G=\Aut(X)$ and let $\frak g=(\frak l+\mathbb C) \rhd U$ be the corresponding Lie algebra. As in Proposition \ref{tangentsp of X}, we give a gradation on the Lie algebra $\frak g$. Let $H:=H(\frak m, G_0)$ be the Lie subgroup of $G$ associated with $\frak h=\bigoplus_{p \geq 0} \frak g_p$ (See (\ref{H})). Let $\frak m=\bigoplus_{p<0}\frak g_{p}$ and let $G_0$ be the Lie subgroup of $G$ corresponding to $\frak g_0$. Let $(M,F)$ be a regular filtered manifold of type $\frak m$. Then, for a given $G_0$-structure on $(M,F)$, there exists a Cartan connection of type $(\frak g, H)$ so that two $G_0$-structures on $(M,F)$ are (locally) equivalent when the associated Cartan connections of type $(\frak g, H)$ are (locally) equivalent. \end{proposition}

\begin{proof}
We will apply Theorem \ref{connection} to $\frak g=(\frak l+\mathbb C) \rhd U$. As in Proposition \ref{tangentsp of X}, the element $E_X=E_{\alpha} \in \frak g_0$ gives the gradation of $\frak g$. By Lemma \ref{fun} (1), $\frak m$ is a fundamental graded Lie subalgebra. By Proposition \ref{Prolongation}, the Lie algebra $\frak g$ is the prolongation of $(\frak m, \frak g_0)$.

By Proposition \ref{imbedding}, let $\tilde{\frak g}=\frak{gl}(V)$ which contains $\frak g$ and $\frak g^*$. The Killing form on $\frak{gl}(V)$ itself is an $\ad(\frak g(\frak m,\frak g_0))$-invariant symmetric bilinear form $(\cdot , \cdot)$ such that the restricted inner product $\{\cdot, \cdot \}$ on $\frak g$ is a positive definite Hermitian inner product. This proves (1) of Theorem \ref{connection}.

We can give a gradation on $\frak{gl}(V)$ by the element $E_X$. Since $V$ is a representation space of $\frak l$, we have decomposition $\frak{gl}(V)=\bigoplus_k\frak{gl}(V)_k$, where $\frak{gl}(V)_k$ is eigenspace of $E_X$ with eigenvalue $k$. We shift this decomposition to give a gradation on $\frak{gl}(V)$ satisfying (2) of Theorem \ref{connection}.

Since $U \subset V_{\alpha}\otimes V^*_{\beta}$, we shift the gradation on $V_{\alpha}\otimes V^*_{\beta}$ to make it the extended gradation of the gradation on $U$. Then, $\frak g_{p} \subset \tilde{\frak g}_{p}$ for any integer $p$.

We also shift the gradation on $U^*$ and extend it to $V^*_{\alpha}\otimes V_{\beta}$ such that $\tau(\frak g_p)\subset\tilde{\frak g}_{-p}$ for $p\geq 0$. More precisely, we can do this as follows: We have $\tau(E_X)=-E_X$, and for $z_1, z_2 \in \frak g$, $[\tau(z_1), \tau(z_2)]=\tau([z_1,z_2])$. Hence, for $z \in \frak g_p$ such that $[E_X,z]=kz$ where $k \in \mathbb Z$, we see that
\begin{eqnarray*}[E_X,\tau(z)]=-[\tau(E_X),\tau(z)]=-\tau([E_X,z])=-k\tau(z).\end{eqnarray*}
In Proposition \ref{tangentsp of X}, if one shifted gradation on $U_k$ by $j$, i.e. $p=k+j$, then one shift gradation on $U^*_{-k}:=(U_k)^*$ by $-j$, and also shift gradation on $V^*_{\alpha}\otimes V_{\beta}$ by $-j$ en bloc.

The remaining conditions (3) and (4) of Theorem \ref{connection} are clear.
\end{proof}

The next proposition is proved in \cite{Hw15} providing the essence of Theorem 4.1 in \cite{BM}.

\begin{proposition}[From Proposition 2.9 of \cite{Hw15}]\label{positive rational curve} Let $M$ be a manifold. Assume that there exists a non-constant holomorphic map $f\colon \mathbb P^1\rightarrow M$ such that $f^* T(M)$ is a positive vector bundle, i.e., $f^*T(M)\cong\mathcal O(a_1)\oplus\cdots\oplus\mathcal O(a_n)$ where $a_i \geq 1$ and $n=\dim M$. Let $H$ be a closed connected subgroup of a connected Lie group $G$. Let $\frak g$ be the Lie algebra of $G$. Then, any Cartan connection on $M$ of type $(\frak g, H)$ is locally flat.
\end{proposition}

\begin{proof}[Sketch of the proof] Given a Cartan connection $\omega$ on a principal $H$-bundle $P \rightarrow M$, we can associate a principal $G$-bundle $\tilde{P} \rightarrow M$ with an Ehresmann connection $\tilde\omega$ as in Section 3 of \cite{BM}. For a curve $f_t \colon \mathbb P^1 \rightarrow M$ with positive ${f_t}^*T(M)$, we see that ${f_t}^*K(\tilde\omega)=0$, where $K(\tilde\omega)$ is the curvature of the connection $\tilde\omega$. We can see the vanishing of that curvature along a curve with positive tangents in the proof of Theorem 3.1 in \cite{Bi}.

By assumption, there is a non-constant holomorphic map $f\colon \mathbb P^1\rightarrow M$ such that $f^* T(M)$ is a positive vector bundle. Then, there exists a family of holomorphic maps
 \begin{eqnarray*}
 \{f_t \colon \mathbb P^1 \rightarrow M | t \in \triangle^k, f_t^*T(M) \text{ positive}\}
 \end{eqnarray*}
 parametrized by a polydisc $\triangle^k$, for some $k>0$ such that the union of their images $\cup_{t \in \triangle^k}f_t(\mathbb P^1)$ contains a nonempty open subset $U$ of $M$.
 For a nonempty open set $U \subset M$, the curvature $K(\tilde\omega)$ vanishes on $U$ as stated above and hence vanishes on the whole space $M$. Thus, $\tilde\omega$ is locally flat on $M$, which implies that $\omega$ is locally flat on $M$. \end{proof}

The following is from Proposition 7.9 of \cite{Hw15}, which is well known from Proposition \textrm{II}.3.7 and Theorem \textrm{IV}.3.7 of \cite{Kol}.

\begin{proposition}\label{codim2} Let $M$ be a uniruled projective manifold of Picard number one. Then, for any subvariety $Z\subset M$ of codimension two, there exists $f:\mathbb P^1\rightarrow M$ with $f(\mathbb P^1)\cap Z=\emptyset$ and $f^*T(M)$ is positive.\end{proposition}

\begin{theorem}\label{Main Theorem} Let $X$ be a smooth nonhomogeneous projective horospherical variety of Picard number one. Let $M$ be a Fano manifold of Picard number one. Then, any geometric structure of $M$ modeled on $X$ is locally equivalent to the standard geometric structure on $X$.
\end{theorem}

\begin{proof}
 Let $M$ be a Fano manifold of Picard number one. A $G_0$-structure on $(M^o,D)$ is given by the geometric structure of $M$ modeled on $X$, where $D$ is a subbundle of $T(M)$ with singularity $\Sing(D)$ and $M^o=M-\Sing(D)$. By Proposition \ref{Cartan connection}, that regular filtered manifold $(M^o,D)$ of type $\frak m$ admits a Cartan connection on $M^o$ of type $(\frak g, H)$.

 Since $M$ is a uniruled projective manifold and the subvariety $\Sing(D)$ has codimension at least two in $M$, by Proposition \ref{codim2}, there is a rational curve $f\colon\mathbb P^1\rightarrow M$ such that $f(\mathbb P^1)\cap \Sing(D)=\emptyset$ and $f^*T(M)$ is positive. We apply Proposition \ref{positive rational curve} to $M^o= M-\Sing(D)$; thus, the Cartan connection on $M^o$ of type $(\frak g, H)$ is locally flat.

 To conclude, a geometric structure of $M$ modeled on $X$ is locally equivalent to the standard geometric structure on $X$.
\end{proof}

\section*{Acknowledgements}
The paper based on author's thesis at Seoul National University. We thank Jaehyun Hong for her kind guidance and discussion throughout this project, Jun-Muk Hwang for his overall advice and comments of the manuscript.

We would like to show our gratitude to the A. Huckleberry, B. Pasquier, B. M. Doubrov, C. Robles and D. The for sharing their knowledge and ways of understanding with us during the visiting Korea Institute of Advanced Study, and also to A.\u{C}ap, J. Slov\'ak and L. Manival at Differential Geometry and its Application 2016 in Brno. We also thank to all reviewers for their thoughtful comments that greatly improved the manuscript.

We also thank to K.-D. Park, M. Kwon and Q. Li for their reading the previous versions of the manuscript.


\begin{thebibliography}{}


\bibitem{Bi}Biswas, Indranil. Principal bundle, parabolic bundle, and holomorphic connection, in A Tribute to C.S. Seshadri. Springer-Verlag, 2003.
\bibitem{BM}Biswas, Indranil, and McKay, Benjamin. Holomorphic Cartan geometries, Calabi-Yau manifolds and rational curves. Differential Geometry and its Applications 28.1 (2010): 102-106.
\bibitem{CS}\u{C}ap, Andreas, and Slov\'ak, Jan. Parabolic Geometries: Background and general theory,
\bibitem{FH}Fulton, William, and Harris, Joe. Representation theory. Vol. 129. Springer-Verlag, 1991.
\bibitem{Ho}Hong, Jaehyun. Fano manifolds with geometric structures modeled after homogeneous contact manifolds. International Journal of Mathematics 11.09 (2000): 1203-1230.
\bibitem{HH}Hong, Jaehyun, and Hwang, Jun-Muk. Characterization of the rational homogeneous space associated to a long simple root by its variety of minimal rational tangents, Algebraic geometry in East Asia-Hanoi 2005. Adv. Stud. Pure Math 50: 217-236.
\bibitem{HM99}Hwang, Jun-Muk, and  Mok, Ngaiming. Varieties of minimal rational tangents on uniruled projective manifolds. Several Complex Variables, MSRI Publications Volume 37, 1999.
\bibitem{HM97}Hwang, Jun-Muk, and  Mok, Ngaiming. Uniruled projective manifolds with irreducible reductive $G$-structures. Journal f\"ur die Reine und Angewandte Mathematik 490 (1997): 55-64.
\bibitem{HM02}Hwang, Jun-Muk, and Mok,  Ngaiming. Deformation rigidity of the rational homogeneous space associated to a long simple root. Annales Scientifiques de l’Ecole Normale Superieure. Vol. 35. No. 2. No longer published by Elsevier, 2002.
\bibitem{HM04}Hwang, Jun-Muk, and Mok, Ngaiming. Deformation rigidity of the 20-dimensional $F_4$-homogeneous space associated to a short root. Springer Berlin Heidelberg, 2004.
\bibitem{Hw15}Hwang, Jun-Muk. Deformation of the space of lines on the 5-dimensional hyperquadric. Preprint.
\bibitem{Jo}Jost, J\"urgen. Riemannian geometry and geometric analysis. Springer-Verlag, 2008.
\bibitem{Kol}Koll\'ar, J\'anos. Rational curves on algebraic varieties. Vol. 32. Springer-Verlag, 1996.
\bibitem{Ko}Kostant, Bertram. Lie algebra cohomology and the generalized Borel-Weil theorem. Annals of Mathematics (1961): 329-387.
\bibitem{Mi}Mihai, Ion Alexandru. odd symplectic flag manifolds. Transformation groups 12.3 (2007): 573-599.
\bibitem{Mok}Mok, Ngaiming. Recognizing certain rational homogeneous manifolds of Picard number 1 from their varieties of minimal rational tangents. AMS-IP Studies in Advanced Mathematics 42 (2008): 41-62.
\bibitem{Mo}Morimoto, Tohru. Geometric structures on filtered manifolds. Hokkaido Math. J. 22 (1993): 263-347.
\bibitem{OV}Onischchik, Arkady L'vovich and Vinberg, Ernest Borisovich. Lie groups and algebraic groups. Springer Series in Soviet Mathematics, Springer-Verlag, 1990.
\bibitem{Pa}Pasquier, Boris. On some smooth projective two-orbit varieties with Picard number 1. Mathematische Annalen 344.4 (2009): 963-987.
\bibitem{RT}Robles, Colleen and The, Dennis. Rigid Schubert varieties in compact Hermitian symmetric spaces. Selecta Mathematica 18.3 (2012): 717-777.
\bibitem{Ta70}Tanaka, Noboru. On differential systems, graded Lie algebras and pseudo-groups. J. Math. Kyoto Univ 10.1 (1970): 1-82.
\bibitem{Ta79}Tanaka, Noboru. On the equivalence problems associated with simple graded Lie algebras. Hokkaido Math. J 8.1 (1979): 23-84.
\bibitem{Wa}Wan, Zhe-xian. Lie algebras (international series of monographs in pure and applied mathematics v.104), Pergamon Press. 1975.
\bibitem{Ya}Yamaguchi, Keizo. Differential systems associated with simple graded Lie algebras. Adv. Stud. Pure Math 22 (1993): 413-494.
\end{thebibliography}
\end{document}